\documentclass[11pt, a4paper, reqno]{amsart}
\usepackage[a4paper, margin=3.75cm]{geometry}
\usepackage{amssymb}
\usepackage{mathrsfs}
\usepackage{graphicx}

\usepackage[width=1.0\textwidth]{caption}

\usepackage[breaklinks]{hyperref}

\hypersetup{
unicode=true,
colorlinks=true,
linkcolor=blue,
citecolor=blue,
urlcolor=blue,
filecolor=blue,
bookmarksnumbered=true,
pdfstartview=FitH,
pdfhighlight=/N
}

\theoremstyle{plain}
\newtheorem{theorem}{Theorem}
\newtheorem{lemma}{Lemma}
\newtheorem{statement}{Statement}
\newtheorem{proposition}{Proposition}
\newtheorem{consi}{Corollary}

\theoremstyle{definition}

\newtheorem{example}{Example}

\numberwithin{equation}{section}

\newcommand\starto{\smash{\raise4pt\rlap{$\mkern19.5mu\scriptscriptstyle\ast$}}\longrightarrow}

\def\({\left(}
\def\){\right)}
\def\[{\left[}
\def\]{\right]}

\def\mdeg{\operatorname{deg}}
\def\const{\operatorname{const}}
\def\GRS{\operatorname{GRS}}
\def\sZ{\mathscr Z}
\def\NN{\mathbb N}
\def\RR{\mathbb R}
\def\CC{\mathbb C}
\def\PP{\mathbb P}
\def\DD{\mathbb D}
\def\TT{\mathbb T}
\def\QQ{\mathbb Q}
\def\HH{\mathscr H}
\def\FF{\mathbf F}
\def\mybfe{\mathbf E}
\def\RS{\mathfrak R}
\def\zz{\mathbf z}
\def\mytt{\mathbf t}
\def\maa{\mathbf a}
\def\aa{\mathbf a}
\def\bb{\mathbf b}
\def\dd{\mathrm{dd}}
\let\eps\varepsilon
\let\pfi\varphi

\def\vv{V}
\def\blambda{{\boldsymbol\lambda}}
\def\bmu{{\boldsymbol\mu}}

\let\leq\leqslant\let\geq\geqslant
\let\ge\geqslant
\let\myh\widehat\let\myt\widetilde\let\myo\overline

\begin{document}

\title{Scalar equilibrium problem and the limit distribution of the zeros of Hermite--Pad\'e polynomials of type~II}

\author[Nikolay~R.~Ikonomov]{N.~R.~Ikonomov}
\address{Institute of Mathematics and Informatics, Bulgarian Academy of Sciences}
\email{nikonomov@math.bas.bg}
\author[Sergey~P.~Suetin]{S.~P.~Suetin}
\address{Steklov Mathematical Institute of Russian Academy of Sciences}
\email{suetin@mi-ras.ru}
\thanks{The research of the second author was carried out with partial
financial support of the Russian Foundation for Basic
Research (grant no.\ 18-01-00764).}

\date{19.09.2019}

\begin{abstract}
The existence of the limit distribution of the zeros of Hermite--Pad\'e polynomials of type~II
for a~pair of functions forming a~Nikishin system
is proved using the scalar equilibrium problem posed on the two-sheeted Riemann surface.

The relation of the results obtained here to some
results of H.~Stahl (1988) is discussed. Results of numerical experiments are presented.
The results of the present paper and those obtained in the earlier paper of the second author
\cite{Sue18},~\cite{Sue19},~\cite{Sue19b}
are shown to be in good accordance with both H.~Stahl's results and with results of numerical experiments.

Bibliography:~\cite{Van06}~titles. \ref{fig_4}~figures.

\end{abstract}

\maketitle

\markright{HERMITE-PAD\'E POLYNOMIALS}

\setcounter{tocdepth}{1}
\tableofcontents


\section{Introduction and statement of the problem}\label{s1}

\subsection{}\label{s1s1}
The present paper continues the studies initiated
by the second author in \cite{Sue18} and~\cite{Sue19b}. In these papers, a~new scalar approach to the
problem on the limit distribution of the zeros of Hermite--Pad\'e polynomials for a~pair of functions
forming a~Nikishin system was proposed and shown to be equivalent to the
traditional vector approach
(see~\cite{Sue19b}, Theorem~1).
We recall that the traditional approach to this problem is based on the solution of
a~vector equilibrium problem in potential theory with a~$2\times2$-matrix
(known as the Nikishin matrix); see,
first of all, \cite{Nik86}, \cite{NiSo88}, \cite{GoRaSo97}, and~\cite{Gon03}, and
also \cite{ApLy10}, \cite{BaGeLo18}, \cite{LoMi18},
and~\cite{LoVa18}, and the references given therein. The vector equilibrium problem
is known to have a~unique solution given by vector measure $\vec{\lambda}=(\lambda_1,\lambda_2)$. Moreover, the limit
distribution of the zeros of Hermite--Pad\'e polynomials of type~I exists and coincides with the measure $\lambda_2$,
while for Hermite--Pad\'e polynomials of type~II the limit distribution of the zeros coincides
with the measure~$\lambda_1$.
The alternative approach of~\cite{Sue18} is that
the equilibrium problem should be considered not on the Riemann sphere $\myh{\CC}$,
but rather on some two-sheeted Riemann surface.
As a~result, the equilibrium problem becomes scalar, which leads to the appearance of another
scalar equilibrium measure~$\blambda$, whose support now lies on the Riemann surface.
In~\cite{Sue18} it was shown that the limit distribution of the zeros
of Hermite--Pad\'e polynomials of type~I coincides with the measure $\lambda=\pi_2(\blambda)$,
where $\pi_2$ is the canonical projection (a~two-sheeted covering) of the Riemann surface
under consideration to the Riemann sphere. In the present paper, the scalar equilibrium measure~$\blambda$
is used to characterize the limit distribution of the zeros of
Hermite--Pad\'e polynomials of type~II for the same pair of functions $f_1$ and $f_2$ as in the papers
\cite{Sue18} and~\cite{Sue19b} (see Section~\ref{s1s2} below). As was already pointed out,
on an example of two Markov functions $f_1$ and $f_2$ considered here (see formula \eqref{1} below)
it was shown in~\cite{Sue19b} that these two approaches (the vector and scalar ones) are equivalent.
The result on the limit distribution of the zeros
of Hermite--Pad\'e polynomials of type~II (Theorem~\ref{th1} below), which is obtained here
in terms of the equilibrium measure~$\blambda$, was proved earlier by
E.~M.~Nikishin~\cite{Nik86} in a~much more general setting using the
traditional vector approach.

The idea of employing the potential theory on a~Riemann surface
for solving the problem on the limit distribution of the zeros of Hermite--Pad\'e polynomials of
multivalued analytic functions is generally not new.
This approach was proposed earlier by H.~Stahl in his two papers \cite{Sta87} and~\cite{Sta88}.
However, Stahl's approach has not been further developed.
The approach proposed by the second author in~\cite{Sue18} is different from that of
Stahl. In particular, as distinct from the second author's papers \cite{Sue18} and~\cite{Sue18b},
Stahl~\cite{Sta87}, \cite{Sta88}
has never considered any extremal problem in potential theory or
any equilibrium problem, even though he used potentials on a~compact Riemann surface. Nevertheless,
it turned out that for the pair of functions $f_1$ and~$f_2$, which is considered
both in~\cite{Sue18} and in the present paper, these two methods give the same answer in the study of the so-called
weak asymptotics of Hermite--Pad\'e polynomials of both type~I and type~II. A~relation between the results
obtained by the authors' and Stahl's methods (the ``third approach'' in Stahl's terms; see \S~9 of \cite{Sta87})
is discussed below in~\S~\ref{s3}.

So, the ultimate purpose of the present paper is to further develop and apply the new scalar approach in the
study of extremal and equilibrium problems that appear naturally when dealing with the
limit behavior of the zeros of Hermite--Pad\'e polynomials.

\subsection{}\label{s1s2}
As in \cite{Sue18} (see also~\cite{Sue19b}), we set
\begin{equation}
f_1(z):=\frac1{(z^2-1)^{1/2}},\quad
f_2(z):=\frac1\pi\int_{-1}^1\frac{h(x)}{z-x}\frac{dx}{\sqrt{1-x^2}},
\quad z\in D:=\myh{\CC}\setminus E,
\label{1}
\end{equation}
where $E:=[-1,1]$ and we choose the branch of the function $(\,\cdot\,)^{1/2}$
such that $(z^2-1)^{1/2}/z\to1$ as $z\to\infty$; for $x\in(-1,1)$ by $\sqrt{1-x^2}$
we mean the positive square root: $\sqrt{a^2}=a$ for $a\geq0$.
Here\footnote{In \S~\ref{s3}, when discussing the relation between our and Stahl's results,
we shall extend the class of function~$h$; see also~\cite{RaSu13},~\cite{Sue18b},~\cite{Sue19}.}
and everywhere in \S\S~\ref{s1} and~\ref{s2} it is assumed that in \eqref{1}\enskip $h=\myh{\sigma}$~is
a~Markov function supported in a~regular compact set $F\subset\RR\setminus E$; i.e.,
\begin{equation}
h(z)=\myh{\sigma}(z):=\int_{F}\frac{d\sigma(t)}{z-t},\quad z\in\myh\CC\setminus
F,
\label{2}
\end{equation}
where $\sigma$ is a~positive Borel measure with support $S(\sigma)$
in~$F$ and such that $S(\sigma)=F$ and $\sigma'(t):=d\sigma/dt>0$
almost everywhere (a.e.) on~$F$ (see \cite{Sue18}, \cite{GoRaSo97},
\cite{Gon03}). These conventions and notation will be kept throughout
the paper. Moreover, we shall assume that the compact set~$F$ consists
of a~finite number of closed intervals, $F=\bigsqcup_{j=1}^p F_j$, and
the convex hull $\myh{F}$ of~$F$ has no common points with the closed
interval~$E$, $\myh{F}\cap E=\varnothing$. For definiteness, we shall
assume that the compact set $\myh{F}$ lies on the real line to the
right of~$E$.

Since $f_1$ can be written as
\begin{equation}
f_1(z)=\frac1\pi\int_{-1}^1\frac1{z-x}\frac{dx}{\sqrt{1-x^2}},\quad z\in
D,
\label{3}
\end{equation}
from \eqref{1},~\eqref{2}, and~\eqref{3} it follows that
$\Delta f_2(x)/\Delta f_1(x)=\myh\sigma(x)$, $x\in(-1,1)$,
where $\Delta f_j(x)$ is the difference of the limit values (the jump)
of the function $f_j$, $j=1,2$, in the upper and lower half-planes,
respectively. It follows that the pair of functions $(f_1,f_2)$ forms a~Nikishin system (for more on such systems, see
\cite{Nik86},~\cite{NiSo88},
and also~\cite{ApBoYa17},~\cite{BaGeLo18},~\cite{LoVa18}, and the references given therein).
Note that in the paper~\cite{Sue18c} an example of a~multivalued analytic function~$f$ is given such that the
pair of functions $f,f^2$ forms a~Nikishin system (under a~minimal extension of the definition of a~Nikishin system compared to the classical one).
There exist classes\footnote{The numerical examples discussed in~\S\,\ref{s3} are related to these classes.}
of multivalued analytic functions
$f$ such that the pair of functions $f,f^2$
can be naturally considered as a~{\it complex} Nikishin system
(see \cite{RaSu13},~\cite{MaRaSu16},~\cite{Sue18b}).
In connection with the new approach of~\cite{Sue18d} to the problem of
efficient continuation of a~given germ
of a~multivalued analytic function,
this fact seems to be one of the main motivations for the study of equilibrium problems associated with complex Nikishin systems.
Moreover,
there are other applications of Hermite--Pad\'e polynomials to the study of actual problems from various fields of theoretical and applied
mathematics; see, for example, \cite{MaTs17}, \cite{Tri18}, \cite{LoMeSz19} and the references found there.
It is the field of Nikishin systems which has received most attention in recent years.
For example, much effort is now concentrated on the treatment of Hermite--Pad\'e polynomials for Nikishin systems on star-like sets
(see \cite{LoMi18}, \cite{LoLo18}, \cite{LoLo19}).
All this clearly shows the relevance of further development of the general theory
of Hermite--Pad\'e polynomials, and in the first instance, of Nikishin systems.

Given an arbitrary $n\in\NN$, we let $\PP_n$ denote the set of all polynomials of degree
$\leq{n}$ with complex coefficients; $\RR_n[\,\cdot\,]$ is the set of all
polynomials with real coefficients and of degree $\leq{n}$ with respect to the corresponding variable.
For an arbitrary polynomial $Q\in\PP_n^{*}:=\PP_n\setminus\{0\}$, by
$\chi(Q)$ we denote the counting measure of the zeros of the polynomial~$Q$ (counting multiplicities),
$$
\chi(Q):=\sum_{\zeta:Q(\zeta)=0}\delta_\zeta,
$$
where $\delta_\zeta$ is the unit measure concentrated at a~point $\zeta$ (the Dirac delta-function).

We now present some well-known facts about the limit distribution of the zeros of Hermite--Pad\'e polynomials
of type~I and~II for the pair of functions $f_1,f_2$ forming a~Nikishin system
(see \cite{NiSo88},~\cite{GoRaSo97},~\cite{Gon03},~\cite{Van06}, and
also~\cite{LoVa18},~\cite{BaGeLo18}, and the references quoted therein).

Given an arbitrary $n\in\NN$, by $q_{2n}\in\RR_{2n}[z]$,
$q_{2n}\not\equiv0$ and $p_{2n,1},p_{2n,2}\in\RR_{2n-1}[z]$ we denote
the Hermite--Pad\'e polynomials of type~II and index\footnote{More
precisely, here the multiindex $(n,n)$ is meant. But since in this
article we are actually limited only to such multiindices, we do not
require the definition of a~general multiindex.}~$n$ for a~given pair
of functions $f_1,f_2$. Namely, these polynomials are defined (not
uniquely) from the following two relations:
\begin{align}
(q_{2n}f_1-p_{2n,1})(z)=O\(\frac1{z^{n+1}}\),\quad z\to\infty,
\label{4}\\
(q_{2n}f_2-p_{2n,2})(z)=O\(\frac1{z^{n+1}}\),\quad z\to\infty.
\label{5}
\end{align}

Let $M_1(E)$ be the class of all unit (positive Borel) measures with supports
in~$E$; let $M_1(F)$ be the analogous class of unit measures with supports in~$F$. Next,
let $g_F(z,\zeta)$ be the Green function for the domain
$\Omega:=\myo\CC\setminus{F}$ with logarithmic singularity at $z=\zeta$, let
\begin{equation}
G^\mu_F(z):=\int_E g_F(z,t)\,d\mu(t),\quad z\in D,
\label{6}
\end{equation}
be the Green potential of the measure $\mu\in M_1(E)$, and let
\begin{equation}
U^\mu(z):=\int_E\log\frac1{|z-t|}\,d\mu(t),\quad
z\in D,
\label{7}
\end{equation}
be the logarithmic potential of the measure~$\mu$.
It is well known (see Chap.~5 of~\cite{NiSo88}, and
also~\cite{GoRaSo97},~\cite{Gon03},~\cite{GoRaSu11})
that there exists a~unique measure $\lambda_E\in M_1(E)$ such that
$S(\lambda_E)=E$ and
\begin{equation}
3U^{\lambda_E}(x)+G^{\lambda_E}_F(x)\equiv w_E=\const
\label{8}
\end{equation}
on~$E$ (the {\it equilibrium relation}).
The measure~$\lambda_E$ is known as the {\it equilibrium} measure with respect to the mixed
Green-logarithmic potential $3U^\mu(x)+G^\mu_F(x)$, $w_E$~is the corresponding {\it equilibrium constant}.

The following result is well known
(see \cite{Nik86},~\cite{NiSo88},~\cite{GoRaSo97},~\cite{Gon03},~\cite{ApLy10}).

\begin{theorem}\label{th0}
Let $f_1,f_2$ be the pair of functions defined in \eqref{1},
where $h(x)=\myh{\sigma}(x)$, $S(\sigma)=F$ and $\sigma'(t)=d\sigma/dt>0$ a.e.\
on~$F$. Given an arbitrary $n\in\NN$, let $q_{2n}$ be the Hermite--Pad\'e polynomials of type~II
defined by \eqref{4}--\eqref{5}. Then
$\mdeg{q_{2n}}=2n$ for each~$n$, the
polynomial $q_{2n}$ is defined uniquely by the normalization $q_{2n}(z)=z^{2n}+\dotsb$,
all the zeros of the polynomial $q_{2n}$ are simple and lie in the interval $(-1,1)$, and moreover,
\begin{equation}
\frac1n\chi(q_{2n})\overset{*}\longrightarrow2\lambda_E
\label{9}
\end{equation}
as $n\to\infty$,
where $\lambda_E\in M_1(E)$ is the equilibrium measure for problem~\eqref{8}.
\end{theorem}

In \eqref{9} and in what follows, by the convergence ``$\overset{*}\longrightarrow$'' we mean the
weak$^*$ convergence in the space of measures $M_1(E)$.


Note that the above result on the limit distribution of the zeros of Hermite--Pad\'e polynomials of type~II
is a~very particular case of a~more general result proved in a~more general setting than that considered in Theorem~\ref{th0}
(see, first of all \cite{Nik86}, \cite{NiSo88}, Ch.~5, \S~7, Theorem~7.1, and
also~\cite{GoRaSo97},~\cite{ApLy10}). All these results have been obtained in the framework of the
traditional vector approach to the problem
on the limit distribution of the zeros of Hermite--Pad\'e polynomials, the
foundations of which was laid by
A.~A.~Gonchar and E.~A.~Rakhmanov in 1981 (see~\cite{GoRa81}).

The purpose of the present paper is to prove, using the new approach proposed in the series of papers
\cite{Sue18},~\cite{Sue19b} and residing in the scalar equilibrium problem on a~Riemann surface,
that under the hypotheses of Theorem~\ref{th0} there exists the limit distribution
of the Hermite--Pad\'e polynomials $q_{2n}$ of type~II without having recourse to the vector equilibrium problem,
but rather employing directly the terms related to the scalar equilibrium problem.

\subsection{}\label{s1s3}
We shall require the following notation and definitions from \cite{Sue18}.

Given $z\in D=\myh{\CC}\setminus E$, we set
\begin{equation}
\pfi(z):=z+(z^2-1)^{1/2};
\label{11}
\end{equation}
this is the inverse of the Joukowsky function (we recall that everywhere in the present paper we
choose the branch of the function $(\cdot)^{1/2}$ such that $(z^2-1)^{1/2}/z\to1$
as $z\to\infty$). The function~$\pfi$ is single-valued meromorphic in the domain~$D$.

Let $\RS_2=\RS_2(w)$ be the Riemann surface of the function
$w^2=z^2-1$, $w(z)=\pm(z^2-1)^{1/2}$, $z\in D$.
We shall assume that a~point $\zz\in\RS_2$ has the form $\zz=(z,w)$. Let $\pi_2\colon\RS_2\to\myh\CC$ be the two-sheeted
covering of the Riemann sphere $\myh{\CC}$ ($\pi_2$~is the canonical projection):
$\pi_2(\zz)=z$.
From the equality $\pfi(\zz)=z+w$ one can define the function~$\pfi$
on the Riemann surface~$\RS_2$. More precisely, we set $\Phi(\zz)=z+w$.
The function $\Phi(\zz)$ thus defined is the natural analytic extension of the function $\pfi(z)$ from the domain
$D\subset\myh{\CC}$ to the entire Riemann surface $\RS_2$.

Let us now define the global partition of the Riemann surface $\RS_2$ into open
sheets $\RS_2^{(0)}$ (the zero sheet) and $\RS_2^{(1)}$ (the first sheet)
by the rule:
$z^{(0)}:=(z,(z^2-1)^{1/2})\in\RS_2^{(0)}$,
$z^{(1)}:=(z,-(z^2-1)^{1/2})\in\RS_2^{(1)}$, $z\in D$.
As usual, the zero sheet $\RS_2^{(0)}$ of the Riemann surface~$\RS_2$ is identified with the
``physical'' domain $D=\myh{\CC}\setminus E$ on the Riemann sphere.
We set
$g_2(\zz):=-\log|\Phi(\zz)|=\log|z-w|$.
From the above partition of the Riemann surface $\RS_2$ into sheets, we have
$\Phi(z^{(0)})=\pfi(z)$, $\Phi(z^{(1)})=1/\pfi(z)$.
Hence $g_2(z^{(0)})=-\log|z|+O(1)$,
$g_2(z^{(1)})=\log|z|+O(1)$, $z\to\infty$, and $g_2(z^{(0)})<g_2(z^{(1)})$.
The above partition of the Riemann surface $\RS_2$ into sheets is a~Nuttall partition
(see \cite{Nut84}, \S\,3, \cite{KoPaSuCh17}, Lemma~5).
Moreover, $\pi_2(\RS_2^{(0)})=\pi_2(\RS_2^{(1)})=D$.

We set $\vv(\zz):=-\log|\Phi(\zz)|=\log|z-w|$ for $\zz\in\RS_2$; in what follows, the function
$\vv(\zz)$ will play the role of an external field\footnote{More precisely, it will play the role of the potential
of an external field; note that the function $V(\zz)$ is harmonic near the compact set~$\FF$.}
in the equilibrium problem considered here.
Let $\FF=F^{(1)}\subset\RS_2$ be the compact set lying on the first sheet
$\RS^{(1)}_2$ of the Riemann surface $\RS_2$
and such that $\pi_2(\FF)=F$.

By $M_1(\FF)$ we denote the space of all
unit (positive Borel) measures with support in~$\FF$.
Following \cite{Sue18} and~\cite{Sue19b}, for an arbitrary measure $\bmu\in M_1(\FF)$,
we introduce the function $P^{\bmu}(\zz)$ of a~point $\zz\in\RS_2$ (the ``potential'' of the measure~$\bmu$,
see Remark~2 in~\cite{Sue19b}),
\begin{equation}
P^{\bmu}(\zz):=\int_{\FF}\log\frac{\left|1-1/\bigl(\Phi(\zz)\Phi(\mytt)\bigr)\right|}
{|z-t|^2}
\,d\bmu(\mytt),\quad \zz\in\RS_2\setminus(F^{(0)}\cup F^{(1)}),
\label{12}
\end{equation}
and consider the corresponding energy of the measure~$\bmu$ (cf.~\cite{Chi18} and~\cite{Chi19})
\begin{equation}
J(\bmu):=
\iint_{\FF\times\FF}\log\frac{\left|1-1/\bigl(\Phi(\zz)\Phi(\mytt)\bigr)\right|}
{|z-t|^2}\,d\bmu(\zz)\,d\bmu(\mytt)
=\int_{\FF} P^\bmu(\zz)\,d\bmu(\zz)
\label{14}
\end{equation}
with respect to the kernel
$$
\log\frac{\left|1-1/\bigl(\Phi(\zz)\Phi(\mytt)\bigr)\right|}{|z-t|^2}.
$$
We also define the energy of the measure $\bmu$ in an external field~$\vv$,
\begin{align}
J_\vv(\bmu):&=
\iint_{\FF\times\FF}\biggl\{
\log\frac{\left|1-1/\bigl(\Phi(\zz)\Phi(\mytt)\bigr)\right|}
{|z-t|^2}
+\vv(\zz)+\vv(\mytt)\biggr\}\,d\bmu(\zz)\,d\bmu(\mytt)\notag\\
&=\int_{\FF} P^\bmu(\zz)\,d\bmu(\zz)+2\int_{\FF}\vv(\zz)\,d\bmu(\zz).
\label{15}
\end{align}

By $M_1^\circ(\FF)$ we denote the
set of all measures $\bmu\in M_1(\FF)$ with finite\footnote{By~\eqref{14},
this set coincides with the set of all probability measures with support in~$F$
and having finite energy with respect to the logarithmic kernel $-\log|z-t|$.}
energy $J_V(\bmu)$. In what follows, we shall identify
the measure $\bmu\in M_1(\FF)$ with the measure $\mu=\pi_2(\bmu)\in M_1(F)$,
where $\pi_2(\bmu)(e):=\bmu(e^{(1)})$ for any $\bmu$-measurable set $e\subset F$.

The following result\footnote{Note that here
we have slightly changed the notation of~\cite{Sue18} making it more appropriate to the
scalar approach based on the use of an appropriate Riemann surface. This
approach was proposed in~\cite{Sue18} and is developed in the present paper.}
combines the principal results of the papers \cite{Sue18} and~\cite{Sue19b}.

\begin{proposition}[\rm (see \cite{Sue18}, Theorem~1;~\cite{Sue19b},
Theorem~1 and Remark~3)]\label{prop1}
In the class $M_1^\circ(\FF)$, there exists a~unique (scalar) measure $\blambda=\lambda_{\FF}\in M_1^\circ(\FF)$
with the following property:
\begin{equation}
J_\vv(\blambda)=\inf_{\bmu\in M_1(\FF)}J_\vv(\bmu).
\label{18}
\end{equation}
The measure $\blambda$ is completely characterized by the following
equilibrium condition\footnote{Here the equilibrium relation on the
compact set~$\FF$ is given in a~slightly different form than
in~\cite{Sue18}. The fact that $P^\blambda(\zz)+\vv(\zz)\equiv w_\FF$
everywhere on~$S(\lambda_\FF)$ follows from the equality
$S(\lambda_\FF)=\FF$ (see~\cite{Sue19b}) and since $\FF$~is regular.}:
\begin{equation}
P^\blambda(\zz)+\vv(\zz)\equiv w_{\FF}, \quad \zz\in \FF=S(\blambda).
\label{19}
\end{equation}
\end{proposition}

\subsection{}\label{s1s4}
In the present paper we employ the scalar equilibrium problem~\eqref{19}
to prove the following

\begin{theorem}\label{th1}
Let $f_1$ and $f_2$ be the functions defined by relations~\eqref{1} and
let $q_{2n}$, $n\in\NN$, be the monic Hermite--Pad\'e polynomials of type~II defined by the two
relations \eqref{4}--\eqref{5}. Then $\mdeg{q_{2n}}=2n$
for all sufficiently large $n$, all the zeros of the polynomial $q_{2n}$ lie in the interval
$(-1,1)$, and moreover,
\begin{multline}
\lim_{n\to\infty}\log\frac1{|q_{2n}(z)|^{1/n}}=
U^{\lambda_F}(z)-\int_F\log\frac1{|\pfi(z)-\pfi(t)|}\,d\lambda_F(t)
-2\log|\pfi(z)|+\log{2},
\\
z\in\CC\setminus{E},
\label{37}
\end{multline}
uniformly inside the domain $\CC\setminus{E}$,
where $q_{2n}(z)=z^{2n}+\dotsb$, $\lambda_F=\pi_2(\lambda_{\FF})$, $\lambda_{\FF}$ is the equilibrium measure
for problem~\eqref{19}.
\end{theorem}

The following corollary can be obtained from Lemma~\ref{lem2} (see \S~\ref{s4} below) and the fact that
$\log|\pfi(z)|=g_E(z,\infty)=\gamma_E-U^{\tau_E}(z)$, where $g_E(z,\infty)$ is the Green
function for the domain~$D$, $\tau_E$~is the Chebyshev measure for the interval~$E$, and
$\gamma_E=\log2$ is the Robin constant for~$E$.

\begin{consi}\label{cons1}
Under the hypotheses of Theorem~\ref{th1},
\begin{equation}
\frac1n\chi(q_{2n})\overset{*}\longrightarrow2\mu,
\quad n\to\infty,
\label{37.2}
\end{equation}
where
\begin{equation}
\mu=\frac14\beta_E(\lambda_F)+\frac34\tau_E.
\label{38}
\end{equation}
\end{consi}

Indeed, representation~\eqref{38} for the measure $\mu$ is obtained by applying the operator
$\dd^c$ to the right-hand side of \eqref{37} and using Lemma~\ref{lem2}.

According to Theorem~1 of \cite{Sue19b}, the right-hand side of \eqref{38} is equal to
$\lambda_1=\lambda_E$ (see Theorem~\ref{th0} above).
Thus, relation~\eqref{37.2} is equivalent to~\eqref{9}.

\section{Proof of Theorem~\ref{th1}}\label{s2}

\subsection{}\label{s2s1}
Let $q_{2n}$ be the Hermite--Pad\'e polynomials of type~II defined
by \eqref{4}--\eqref{5}.
Let $\gamma_1$ be a~closed contour separating the closed interval~$E$ from the compact set~$F$
and the point at infinity $z=\infty$. From \eqref{4} it follows that the
following orthogonality relations hold:
\begin{equation}
\int_{\gamma_1}q_{2n}(\zeta)\zeta^kf_1(\zeta)\,d\zeta=0,
\quad k=0,\dots,n-1.
\label{61}
\end{equation}
Therefore, for any Chebyshev polynomial (of the first kind)
$T_k(\zeta)=\pfi^k(\zeta)+\pfi^{-k}(\zeta)=2^k\zeta^k+\dotsb$, $k=0,\dots,n-1$, we have
\begin{equation}
\int_{\gamma_1}q_{2n}(\zeta)T_k(\zeta)f_1(\zeta)\,d\zeta=0.
\label{62}
\end{equation}
Since $\Delta f_1(x)=-2i/\sqrt{1-x^2}$ for $x\in(-1,1)$, we get from \eqref{62} that
\begin{equation}
\int_{-1}^1q_{2n}(x)T_k(x)\,\frac{dx}{\sqrt{1-x^2}}=0, \quad k=0,\dots,n-1.
\label{63}
\end{equation}
Therefore, we have the representation
\begin{equation}
q_{2n}(x)=\sum_{j=n}^{2n}c_jT_j(x);
\label{64}
\end{equation}
here $c_j\in\RR$, because $q_{2n}$ is a real polynomial. Now an appeal to \eqref{5} shows that
\begin{equation}
\int_{\gamma_1}q_{2n}(\zeta)p(\zeta)f_2(\zeta)\,d\zeta=0,
\label{65}
\end{equation}
where $p\in\PP_{n-1}$ is an arbitrary polynomial. Using representation~\eqref{64}
for the polynomial $q_{2n}$, we have from \eqref{65}
\begin{equation}
\int_{\gamma_1}\biggl\{\sum_{j=n}^{2n}c_jT_j(\zeta)\biggr\}p(\zeta)f_2(\zeta)\,
d\zeta=0,
\label{66}
\end{equation}
or, equivalently,
\begin{equation*}
\int_{-1}^1
\biggl\{\sum_{j=n}^{2n}c_jT_j(x)\biggr\}p(x)\Delta f_2(x)\,dx=0.
\end{equation*}
Hence, using representation~\eqref{1} for the function $f_2(x)$, we have,
for any polynomial $p\in\PP_{n-1}$,
\begin{equation}
\int_{-1}^1
\biggl\{\sum_{j=n}^{2n}c_jT_j(x)\biggr\}p(x)\myh{\sigma}(x)
\,\frac{dx}{\sqrt{1-x^2}}=0.
\label{68}
\end{equation}

Given $z\in D$, let
\begin{equation}
H_n(z):=\frac1\pi\int_{-1}^1\frac{T_n(x)}{z-x}\frac{dx}{\sqrt{1-x^2}}
=\frac1{\pi T_n(z)}\int_{-1}^1 \frac{T^2_n(x)}{z-x}\frac{dx}{\sqrt{1-x^2}},
\quad n=0,1,\dots,
\label{69}
\end{equation}
be the functions of the second kind corresponding to the Chebyshev polynomials $T_n$. It is well known that
the Chebyshev polynomials $T_n(z)=2^nz^n+\dotsb$ satisfy the following recursion relation:
\begin{equation}
y_n=2zy_{n-1}-y_{n-2},\quad n=1,2,\dots,\quad\text{where}\quad
y_{-1}=0, \ \ y_0=1.
\label{70}
\end{equation}
From \eqref{69} it follows that this relation is also satisfied by the functions of the second kind
$H_n$ (but with different initial conditions: $H_{-1}(z)\equiv1$, $H_0(z)=f_1(z)=1/(z^2-1)^{1/2}$). Moreover, using
\eqref{69} we find that, for $x\in(-1,1)$,
\begin{equation}
\Delta H_n(x):=H_n(x+i0)-H_n(x-i0)=T_n(x)\frac{2}{i\sqrt{1-x^2}}.
\label{71}
\end{equation}
From \eqref{71} it easily follows that relation~\eqref{68} is equivalent to the relation
\begin{equation}
\int_{\gamma_1}\biggl\{\sum_{j=n}^{2n}c_jH_j(\zeta)\biggr\}p(\zeta)\myh{\sigma}(\zeta)
\,d\zeta=0,
\label{72}
\end{equation}
where $\gamma_1$ is an arbitrary contour separating the closed interval~$E$ from the point
at infinity and the compact set~$F$, besides, $p\in\PP_{n-1}$ is a~polynomial, and (see \eqref{2}),
$$
\myh{\sigma}(z)=\int_F\frac{d\sigma(t)}{z-t},\quad z\notin F.
$$
By the second equality in \eqref{69},
$$
H_j(z)=O\(\frac1{z^{j+1}}\),\quad z\to\infty,
$$
and hence \eqref{72} can be transformed to read
\begin{equation*}
\int_{\gamma_2}\biggl\{\sum_{j=n}^{2n}c_jH_j(\zeta)\biggr\}p(\zeta)\myh{\sigma}(\zeta)
\,d\zeta=0,
\end{equation*}
where $\gamma_2$ is an arbitrary closed contour separating the compact set~$F$ from the
closed interval~$E$ and the point at infinity. From the last relation we find that
\begin{equation}
\int_F
\biggl\{\sum_{j=n}^{2n}c_jH_j(t)\biggr\}p(t)\,d\sigma(t)=0
\label{73}
\end{equation}
for an arbitrary polynomial $p\in\PP_{n-1}$ and some real numbers
$c_j$, $j=n,\dots,2n$.

Now, assume for definiteness that $n=2m-1$ is an odd number (the case of even
$n=2m$ is dealt with similarly). From the recursion relation~\eqref{70} it easily follows that
\begin{equation}
\sum_{j=n}^{2n}c_jH_j(z)=q_{n,1}(z)H_{n+m-1}(z)+q_{n,2}(z)H_{n+m}(z),
\label{74}
\end{equation}
and
\begin{equation}
q_{2n}(z)=\sum_{j=n}^{2n}c_jT_j(z)=q_{n,1}(z)T_{n+m-1}(z)+q_{n,2}(z)T_{n+m}(z)
\label{75}
\end{equation}
(cf.~\cite{LoLo18}, Proposition~2.12, Lemma 5.1, and \cite{LoLo19}),
where the polynomials $q_{n,1},q_{n,2}$ lie in $\RR_{m-1}[z]$ (we recall that $n=2m-1$).
It is worth pointing out that in \eqref{74} and~\eqref{75} the polynomials $q_{n,1}$ and $q_{n,2}$ are the same.

So, relation~\eqref{73} assumes the form
\begin{equation}
\int_F
\bigl\{q_{n,1}(t)H_{n+m-1}(t)+q_{n,2}(t)H_{n+m}(t)\bigr\}p(t)\,d\sigma(t)=0
\label{77}
\end{equation}
for an arbitrary polynomial $p\in\RR_{n-1}[t]$.
Using \eqref{74}, the equality $\pfi'(z)/\pfi(z)=1/(z^2-1)^{1/2}$, and the well known fact that
for the Chebyshev polynomials the corresponding functions of the second kind read as
\begin{equation}
H_n(z)=\frac{\varkappa_n\pfi'(z)}{\pfi^{n+1}(z)}
=\frac{\varkappa_n}{\pfi^n(z)(z^2-1)^{1/2}},
\label{76}
\end{equation}
where the constant $\varkappa_n\neq0$, the orthogonality relation \eqref{77} assumes the
form
\begin{equation}
\int_F
\biggl\{q_{n,1}(t)+\frac{q_{n,2}(t)}{\pfi(t)}\biggr\}\frac{\pfi'(t)}{\pfi^{n+m}(t)}
p(t)\,d\sigma(t)=0.
\label{78}
\end{equation}

\subsection{}\label{s2s2}
We now introduce the function
$$
g_n(\zz):=q_{n,1}(z)+q_{n,2}(z)\Phi(\zz),\quad \zz\in\RS_2,
$$
and rewrite relation~\eqref{78} in the form
\begin{equation}
\int_{F^{(1)}}
g_n(t^{(1)})\Phi(t^{(1)})^{n+m}\pfi'(t)
p(t)\,d\sigma(t)=0
\label{79}
\end{equation}
or, equivalently,
\begin{equation}
\int_{F^{(1)}}
g_n(t^{(1)})\Phi(t^{(1)})^{n+m-1}\frac{p(t)}{w(t^{(1)})}\,d\sigma(t)=0.
\label{80}
\end{equation}

The function $g(\zz)$ is meromorphic on the Riemann surface
$\RS_2$ defined above.
Recall that in view of the above partition of the Riemann surface
 $\RS_2$ into sheets, we have
$\Phi(z^{(0)})=\pfi(z)$, $\Phi(z^{(1)})=1/\pfi(z)$.
Since the polynomials $q_{n,1},q_{n,2}$ lie in $\PP_{m-1}$, the function
$g_n(\zz)$ has a~pole of order $\ell_0\leq m$ at the point $\zz=\infty^{(0)}$ and has
a~pole of order $\ell_1\leq m-1$ at the point $\zz=\infty^{(1)}$. The function $g_n(\zz)$ has no other poles
on the Riemann surface $\RS_2$. From Abel's theorem it follows that the function
$g_n(\zz)$ has $\ell_0+\ell_1=\ell_2\leq2m-1=n$ ``free'' (depending on~$n$) zeros at
some points $\bb_{1,n},\dots,\bb_{n,\ell_2}$ on the Riemann surface $\RS_2$.

Since the polynomials $q_{n,1}$ and $q_{n,2}$ are real, the integrand in the orthogonality relation~\eqref{79}
is real on
$\myh{F}^{(1)}$. Since $p\in\PP_{n-1}$ is an arbitrary polynomial, from \eqref{80}
it follows that the integrand in \eqref{80}
should have at least $n$~sign changes on the compact set $\myh{F}^{(1)}$.
Therefore, the function $g_n(\zz)$ should have at least~$n$ simple zeros
on the compact set $\myh{F}^{(1)}$.
As a~result, $\ell_2=n$, $\ell_0=m$, $\ell_1=m-1$, and all the zeros
$\bb_{n,1},\dots,\bb_{n,n}$ of the function $g_n(\zz)$ lie in the compact set
$\myh{F}^{(1)}$.
Consequently, the divisor $(g_n)$ of the function $g_n(\zz)$ has the following form
(cf.\ \cite{Sue18}, formula (3.15)):
\begin{equation}
(g_n)=-m\infty^{(0)}-(m-1)\infty^{(1)}+\sum_{j=1}^{n}\bb_{n,j}.
\label{81}
\end{equation}

Note that in~\cite{Sue18} in relation~(3.15), which has the form
\eqref{81}, it was assumed in general that some of the points
$\aa_{n,j}$ can coincide with the point $\infty^{(0)}$ or with the point
$\infty^{(1)}$. In this case, in the representation like formula (3.15)
of~\cite{Sue18}, which is analogous to~\eqref{81}, terms are canceled
out. However, as was shown above, this is not the case for
Hermite--Pad\'e polynomials of type~II considered here; i.e., we
always have $\bb_{n,j}\neq\infty^{(0)},\infty^{(1)}$.

Since the function $g_n(\zz)$ is meromorphic on the Riemann surface $\RS_2$ and since the genus
of the Riemann surface $\RS_2$ is zero,
the function $g_n$ can be completely recovered from its divisor (of zeros and poles)
up to a~nontrivial multiplicative constant.
Hence (cf.\ \cite{Sue18}, formula (3.16)), we have
\begin{equation}
g_n(\zz)=C_n\prod_{j=1}^n\bigl[\Phi(\zz)-\Phi(b^{(1)}_{n,j})\bigr]
\cdot \Phi(\zz)^{-(m-1)},\quad C_n\neq0.
\label{81.2}
\end{equation}
Indeed, it can be easily seen that the divisor of the function on the right of \eqref{81.2}
coincides with the divisor \eqref{81} of the function
$g_n(\zz)$.

We have
$T_n(z)=\pfi(z)^n+\pfi(z)^{-n}=\Phi(z^{(0)})^n+\Phi(z^{(0)})^{-n}
=\Phi(z^{(0)})^n+\Phi(z^{(1)})^n$
for Chebyshev polynomials, and hence from \eqref{75} we have
\begin{equation}
q_{2n}(z)=g_n(z^{(0)})\Phi(z^{(0)})^{n+m-1}+g_n(z^{(1)})\Phi(z^{(1)})^{n+m-1}
\label{82}
\end{equation}
for Hermite--Pad\'e polynomials of type~II.

Now using the identity (see \cite{GoSu04},~\cite{Sue18})
\begin{equation}
z-a\equiv -\frac{\bigl(\Phi(\zz)-\Phi(\maa)\bigr)
\bigl(1-\Phi(\zz)\Phi(\maa)\bigr)}
{2\Phi(\zz)\Phi(\maa)},
\quad z,a\in D,
\label{82.2}
\end{equation}
employing the relation
$\pfi'(z)=\pfi(z)/(z^2-1)^{1/2}=-1/(\Phi(z^{(1)})w(z^{(1)}))$, and taking into account the convention
$n=2m-1$, we transform equality \eqref{81.2} to read
\begin{equation*}
g_n(\zz)
=\myt{C}_n\prod_{j=1}^n\frac{z-b_{n,j}}{1-\Phi(\zz)\Phi(b^{(1)}_{n,j})}
\cdot\Phi(\zz)^{m-1}
\end{equation*}
(cf.\ formula~(3.18) in~\cite{Sue18}). Therefore,
\begin{equation}
g_n(\zz)\Phi(\zz)^{n+m}
=\myt{C}_n\prod_{j=1}^n\frac{z-b_{n,j}}{1-\Phi(\zz)\Phi(b^{(1)}_{n,j})}
\cdot\Phi(\zz)^{2n+1}.
\label{82.3}
\end{equation}
Now the orthogonality relation \eqref{79} assumes the form
(cf.\ \cite{Sue18}, formula~(3.21))
\begin{equation}
\int_{F^{(1)}}\prod_{j=1}^n\frac{t-b_{n,j}}{1-\Phi(t^{(1)})\Phi(b^{(1)}_{n,j})}
\cdot \Phi(t^{(1)})^{2n+1}\pfi'(t)p(t)\,d\sigma(t)=0
\label{86}
\end{equation}
for an arbitrary polynomial $p\in\PP_{n-1}$, where all $b_{n,j}\in(c,d)$,
$[c,d]=\myh{F}\subset\RR_{+}\setminus{E}$.

Relations \eqref{78} and~\eqref{86} are quite similar to relations~(3.16)
and~(3.21) from \cite{Sue18}, respectively. The only difference is that relation (3.21)
from \cite{Sue18} gives a~formula in which $Q_{n,2}\in\RR_n[\,\cdot\,]$ (a~Hermite--Pad\'e polynomial of type~I)
is {\it fixed}, while $\aa_{n,j}$ are
{\it arbitrary} points. The other way round, relation~\eqref{86} involves
an {\it arbitrary} polynomial $p\in\PP_{n-1}$, while all the points
$\bb_{n,j}=b^{(1)}_{n,j}$ are {\it fixed}, because they correspond to the {\it fixed}
Hermite--Pad\'e polynomial of type~II $q_{2n}$ considered here. (Recall that for this polynomial
the explicit representation~\eqref{82} in terms of the function $g_n(\zz)$ is satisfied.)

\subsection{}\label{s2s3}
We set (cf.\ \cite{LoLo18}, Definition~2.11, \cite{LoLo19}, Definition~2.3)
\begin{equation}
P_n(z):=\prod_{j=1}^n(z-b_{n,j}).
\label{87.2}
\end{equation}
Exactly as in~\cite{Sue18}, one can show that there exists the limit
(as $n\to\infty$) distribution of the zeros of the polynomials $P_n$.
Namely, the following result holds.

\begin{lemma}\label{lem1}
If $n\to\infty$, then
\begin{equation}
\frac1n\chi(P_n)\overset{*}\longrightarrow\lambda_F,
\label{88}
\end{equation}
where $\lambda_F=\pi_2(\lambda_{\FF})$ and the polynomials $P_n$ are defined by representation~\eqref{87.2}.
\end{lemma}

\begin{proof}
Lemma~\ref{lem1} is proved by the $\GRS$-method (see \cite{GoRa87},
\cite{Rak18},~\cite{Sue18}).
As usual, when applying the $\GRS$-method\footnote{Note that, under the hypotheses of
Theorem~\ref{th1} and Lemma~\ref{lem1}, the
$\GRS$-method is much easier to
deal with, because the $S$-compact set $F$~is a~finite union of closed intervals of the real line and $\sigma$~is a~positive measure on~$F$; {cf}.\
\cite{GoRa87}, \cite{RaSu13}, \cite{Rak18},~\cite{Sue18}.},
we assume the contrary, i.e.,
\begin{equation}
\frac1n\chi(P_n)\not\to\lambda=\lambda_F
\label{039.2}
\end{equation}
as $n\to\infty$.
We shall arrive at a contradiction by using the orthogonality conditions \eqref{86} and
assumption \eqref{039.2}.

It is known that all the zeros of the polynomial $P_n$ lie in the closed interval $\myh{F}$.
Moreover, by the orthogonality conditions~\eqref{86}, each gap between
the closed intervals $F_j$ contains at most one zero of the polynomial $P_n$.
The weak compactness of the space of measures $M_1(\myh{F})$, $F=\bigsqcup_{j=1}^p
F_j$, shows that
\begin{equation}
\frac1n\chi(P_n)\to\mu\neq\lambda,
\quad n\in\Lambda,\quad n\to\infty,
\label{040}
\end{equation}
for some infinite subsequence $\Lambda\subset\NN$; besides, $S(\mu)\subset{F}$,
$\mu\in M_1(F)$, $\mu(1)=1$.
We claim that relation~\eqref{040} and the orthogonality condition~\eqref{86} contradict each other.

Setting
$$
\myt{U}^\mu(z):=\int_F\log\frac1{|1-\pfi(z)\pfi(t)|}\,d\mu(t),
$$
we have (see \eqref{12})
$$
P^\mu(z)=2U^\mu(z)-\myt{U}^\mu(z).
$$
Since $\mu\neq\lambda$, we have, for $z\in S(\mu)\subset F$,
\begin{equation}
P^\mu(z)+V(z)\not\equiv m_0
:=\min_{z\in F}\bigl(P^\mu(z)+V(z)\bigr)=P^\mu(t_0)+V(t_0)
\label{041}
\end{equation}
(see Proposition~\ref{prop1}),
where $t_0\in F$ and $V(z):=V(z^{(0)})$. Therefore, there exist a~point $t_1\in S(\mu)$, $t_1\neq
t_0$, and $\eps>0$ such that
\begin{equation}
P^\mu(t_1)+V(t_1)=m_1>m_0+\eps.
\label{041.2}
\end{equation}
Since the function $V(z)$ is harmonic and the potential $P^\mu(z)$ is
lower semicontinious, the same inequality \eqref{041.2} also holds in
some $\delta$-neighborhood
$U_\delta(t_1):=(t_1-\delta,t_1+\delta)\not\ni t_0$, $\delta>0$, of the
point $t_1$. We have $t_1\in S(\mu)$, and hence $\mu(U_\delta(t_1))>0$.
Therefore, for all sufficiently large $n\geq n_0$, $n\in\Lambda$, there
exists a~polynomial $p_n(z)=(z-\zeta_{n,1})(z-\zeta_{n,2})$ such that
$\zeta_{n,1},\zeta_{n,2}\in U_\delta(t_1)$ and $p_n$ divides the
polynomial $P_n$; i.e., $P_n/p_n\in\PP_{n-2}$. We set
\begin{equation}
\myt{P}_n(z):=\frac{P_n(z)}{p_n(z)}=\prod_{j=1}^{n-2}(z-b_{n,j})
\label{041.3}
\end{equation}
in representation \eqref{041.3}; moreover, in what follows we assume that the points
$b_{n,j}$ are labeled so that
$\zeta_{n,1}=b_{n,n-1}$ and $\zeta_{n,2}=b_{n,n}$.

If we put $p=\myt{P}_n$ in the orthogonality conditions~\eqref{86}, then
\eqref{86} assumes the form
\begin{align}
0&=
\int_{F\setminus{U_\delta(t_1)}}
\frac{P_n^2(t)}{p_n(t)}\prod_{j=1}^{n}\frac1{1-\pfi(t)\pfi(b_{n,j})}
\cdot\frac{\pfi'(t)}{\pfi^{n+1}(t)}\,d\sigma(t)\notag\\
&+
\int_{\myo{U}_\delta(t_1)}
\frac{P_n^2(t)}{p_n(t)}\prod_{j=1}^{n}\frac1{1-\pfi(t)\pfi(b_{n,j})}
\cdot\frac{\pfi'(t)}{\pfi^{n+1}(t)}\,d\sigma(t).
\label{042}
\end{align}
We denote by $I_{n,1}$ and
$I_{n,2}$,
respectively, the first and second integrals in \eqref{042}, and define
\begin{equation}
\prod{\vphantom{H^H}}'\frac{x-b_{n,j}}{1-\pfi(t)\pfi(b_{n,j})}
:=\prod_{j=1}^{n-2}\frac{x-b_{n,j}}{1-\pfi(t)\pfi(b_{n,j})}
\cdot\frac1{(1-\pfi(t)\pfi(b_{n,n-1}))(1-\pfi(t)\pfi(b_{n,n}))}.
\label{042.2}
\end{equation}
Since the integrand
in~$I_{n,1}$ has constant sign for $t\in F\setminus{U_\delta(t_1)}$, we have
\begin{align}
|I_{n,1}|
&=
\int_{F\setminus{U_\delta(t_1)}}\biggl|
\frac{P_n^2(t)}{p_n(t)}\prod_{j=1}^{n}\frac1{1-\pfi(t)\pfi(b_{n,j})}
\cdot\frac{\pfi'(t)}{\pfi^{n+1}(t)}
\biggr|\,d\sigma(t)\notag\\
&=\int_{F\setminus{U_\delta(t_1)}}
|P_n(t)|
\biggl|\prod{\vphantom{H^H}}'\frac{x-b_{n,j}}{1-\pfi(t)\pfi(b_{n,j})}
\biggr|
\cdot\frac{\pfi'(t)}{\pfi^{n+1}(t)}\,d\sigma(t).
\label{043}
\end{align}

A similar analysis (see Lemma~7 of \cite{GoRa87})
with the use of standard machinery of the
logarithmic potential theory shows that
\begin{equation}
\lim_{\substack{n\to\infty\\n\in\Lambda}} |I_{n,1}|^{1/n}
=\exp\biggl\{
-\min_{t\in F\setminus{U_\delta(t_1)}}\bigl(P^\mu(t)+V(t)
\bigr)\biggr\}
=e^{-m_0}.
\label{044}
\end{equation}
For completeness of presentation, we give here the proof of the limit relation~\eqref{044}
(cf.\ Lemma~7 in~\cite{GoRa87}).

Indeed, we have, as $n\to\infty$,
\begin{equation}
-\frac1n\sum_{j=1}^{n-2}\log|1-\pfi(t)\pfi(b_{n,j})|
\to\int_F\log\frac1{|1-\pfi(t)\pfi(t)|}\,d\mu(t)=\myt{U}^\mu(t)
\label{071}
\end{equation}
uniformly with respect to $t\in F$. Therefore, as $n\to\infty$
\begin{align}
\min_{t\in F}\biggl\{-\frac1n
\log\biggl(
|P_n(t)|
\biggl|\prod{\vphantom{H^H}}'\frac{x-b_{n,j}}{1-\pfi(t)\pfi(b_{n,j})}
\biggr|
\cdot\frac{\pfi'(t)}{\pfi^{n+1}(t)}
\biggr)\biggr\}
\to\min_{t\in F} \bigl\{P^\mu(t)+V(t)\bigr\}.
\label{072}
\end{align}
It follows that
\begin{equation}
\max_{t\in F}\biggl\{|P_n(t)|
\biggl|\prod{\vphantom{H^H}}'\frac{x-b_{n,j}}{1-\pfi(t)\pfi(b_{n,j})}
\biggr|
\cdot\frac{\pfi'(t)}{\pfi^{n+1}(t)}\biggr\}^{1/n}
\to\exp\bigl\{-\min_{t\in F}\bigl[P^\mu(t)+V(t)\bigr]\bigr\}
\label{073}
\end{equation}
as $n\to\infty$. Hence we have the upper estimate
\begin{equation*}
\varlimsup_{\substack{n\to\infty\\n\in\Lambda}} |I_{n,1}|^{1/n}
\leq e^{-m_0}.
\end{equation*}
Let us prove the corresponding lower estimate. The potential $P^\mu(z)$ is weakly continuous,
and hence, the function $P^\mu(z)+V(z)$
is approximately continuous with respect to the Lebesgue measure on the compact set~$F$. Consequently,
for any $\eps>0$, the set
$$
e=\{t\in F: (P^\mu+V)(t)<m_0+\eps\}
$$
has positive Lebesgue measure. From our assumptions we have
$$
-\frac1n\log
\biggl\{|P_n(t)|
\biggl|\prod{\vphantom{H^H}}'\frac{x-b_{n,j}}{1-\pfi(t)\pfi(b_{n,j})}
\biggr|
\cdot\frac{\pfi'(t)}{\pfi^{n+1}(t)}\biggr\}
\to (P^\mu+V)(t)
$$
as $n\to\infty$ with respect to the measure on~$F$. So, the measure of the set
\begin{equation*}
e_n:=\biggl\{x\in e:
-\frac1n\log
\biggl(|P_n(t)|
\biggl|\prod{\vphantom{H^H}}'\frac{x-b_{n,j}}{1-\pfi(t)\pfi(b_{n,j})}
\biggr|
\cdot\frac{\pfi'(t)}{\pfi^{n+1}(t)}\biggr)
<m_0+\eps\biggr\}
\end{equation*}
tends to the measure of~$s$ as $n\to\infty$. Therefore,
\begin{equation}
\varliminf_{\substack{n\to\infty\\n\in\Lambda}} |I_{n,1}|^{1/n}
\geq e^{-(m_0+\eps)}\lim_{\substack{n\to\infty\\n\in\Lambda}}
\(\int_{e_n}\pfi'(t)\,d\sigma(t)\)^{1/n}=e^{-(m_0+\eps)};
\label{075}
\end{equation}
the last equality in \eqref{075} holding because $\sigma'(t)>0$
a.e.\ on~$F$. The lower estimate
$$
\varliminf_{\substack{n\to\infty\\n\in\Lambda}} |I_{n,1}|^{1/n}
\geq e^{-m_0}
$$
follows from~\eqref{075}, because $\eps>0$ is arbitrary.
This proves~\eqref{044}.

On the other hand, for the second integral $I_{n,2}$ we have the estimate
\begin{equation}
|I_{n,2}|\leq
\int_{\myo{U}_\delta(t_1)}
|P_n(t)|
\biggl|\prod{\vphantom{H^H}}'\frac{x-b_{n,j}}{1-\pfi(t)\pfi(b_{n,j})}
\biggr|
\cdot\frac{\pfi'(t)}{\pfi^{n+1}(t)}\,d\sigma(t).
\label{045}
\end{equation}
Now an analysis similar to that above shows that
\begin{equation}
\varlimsup_{\substack{n\to\infty\\n\in\Lambda}} |I_{n,2}|^{1/n}
\leq\exp\biggl\{-\min_{x\in \myo{U}_\delta(t_1)}
\bigl(P^\mu(t)+V(t)\bigr)
\biggr\}
\leq e^{-m_1}<e^{-(m_0+\eps)}.
\label{046}
\end{equation}
But relations~\eqref{044} and~\eqref{046} contradict the equality $I_{n,1}=-I_{n,2}$,
which follows from the orthogonality conditions \eqref{86}.

Lemma~\ref{lem1} is proved.

\end{proof}

\subsection{}\label{s2s4}
To complete the proof of Theorem~\ref{th1} and derive
an explicit strong
asymptotic formulas for the Hermite--Pad\'e polynomials of type~II\enskip $q_{2n}$,
we employ representation \eqref{82} for these polynomials:
$q_{2n}(z)=g_n(z^{(0)})\Phi(z^{(0)})^{n+m-1}+g_n(z^{(1)}\Phi(z^{(1)})^{n+m-1}$
(we recall that $n=2m-1$ by the assumption).
In view of this representation, the zeros of the polynomial $q_{2n}$ (which total~$2n$)
can be derived from the relation
\begin{equation*}
\frac{g_n(z^{(0)})\Phi(z^{(0)})^{n+m-1}}
{g_n(z^{(1)})\Phi(z^{(1)})^{n+m-1}}=-1,
\end{equation*}
or equivalently, by using \eqref{81.2} and~\eqref{82.2},
\begin{equation}
\prod_{j=1}^n\frac{\Phi(\zz)-\Phi(b^{(1)}_{n,j})}
{1-\Phi(\zz)\Phi(b^{(1)}_{n,j})}\cdot\Phi(\zz)^{3n}=-1.
\label{92}
\end{equation}
Changing the variable to $\Phi(\zz)=\zeta$ in \eqref{92}, we get
$|\zeta|>1$ for $\zz\in\RS_2^{(0)}$ and $|\zeta|<1$ for $\zz\in\RS_2^{(1)}$. We set
$\beta_{n,j}:=\Phi(\bb_{n,j})=\Phi(b^{(1)}_{n,j})\in\DD\cap\RR$, where
$\DD:=\{z:|z|<1\}$.
Now \eqref{92} is equivalent to the equation
\begin{equation}
\prod_{j=1}^n
\frac{\zeta-\beta_{n,j}}{1-\beta_{n,j}\zeta}\cdot\zeta^{3n}=-1.
\label{93}
\end{equation}
It is easily seen that the following facts are true:

1) for any $\rho$, $1<\rho<\min(1/|\beta_{n,j}|)$,
the variation in the argument of the function on the left of \eqref{93}
along the circle of radius~$\rho$ is $4n$;

2) $\zeta=\eta$ is a~solution of equation~\eqref{93} if and only if
this equation is satisfied by $\zeta=\myo{\eta}$;

3) $\zeta=\eta$ is a solution of equation~\eqref{93} if and only if
this equation is satisfied by $\zeta=1/\eta$;

4) $\zeta=\pm1$ are not solutions of equation \eqref{93};

5) for $\zeta\in\DD$ we have
$$
\biggl|\prod_{j=1}^n\frac{\zeta-\beta_{n,j}}{1-\beta_{n,j}\zeta}
\cdot\zeta^{3n}\biggr|<1.
$$

From the above facts 1)--5) it follows that
equation~\eqref{93} has precisely
$4n$~roots (counting multiplicities) on the unit circle $\TT:=\partial\DD$.
Moreover, all the solutions of equation \eqref{93} have the form:
$\zeta_{j,\pm}= e^{\pm\theta_j}$, $j=1,\dots,2n$, $\theta_j\in(0,\pi)$.
Consequently, the polynomial
$q_{2n}$ has precisely $2n$ zeros in the interval $(-1,1)$. Therefore, $\mdeg{q_{2n}}=2n$.

Now, we have $|\Phi(z^{(1)})|<1<|\Phi(z^{(0)})|$, and hence from \eqref{82}
and~\eqref{82.2} we get
\begin{equation}
q_{2n}(z)=g_n(z^{(0)})\Phi(z^{(0)})^{n+m-1}\bigl\{1+\delta_n(z)\bigr\}
\label{94}
\end{equation}
for $z\in\DD$ as $n\to\infty$, where $\delta_n(z)$ reads as
$$
\delta_n(z):=\frac{g_n(z^{(1)})\Phi(z^{(1)})^{n+m-1}}
{g_n(z^{(0)})\Phi(z^{(0)})^{n+m-1}}
=\prod_{j=1}^n\frac{\Phi(z^{(1)})-\Phi(b^{(1)}_{n,j})}
{1-\Phi(z^{(1)})\Phi(b^{(1)}_{n,j})}\cdot\Phi(z^{(1)})^{3n}.
$$
Since $\Phi(z^{(1)}),\Phi(b^{(1)}_{n,j})\in\DD$, we have
$$
\biggl|\prod_{j=1}^n\frac{\Phi(z^{(1)})-\Phi(b^{(1)}_{n,j})}
{1-\Phi(z^{(1)})\Phi(b^{(1)}_{n,j})}\biggr|<1
$$
and hence, $\delta_n(z)$ as $n\to\infty$
converges geometrically to zero uniformly (on compact subsets) inside the domain $D=\myo{\CC}\setminus{E}$.

Now using representation~\eqref{82.3}, we get from \eqref{94},
\begin{align}
-\frac1n\log|q_{2n}(z)|
\to U^{\lambda_F}(z)&-\int_F\log\frac1{|\pfi(z)-\pfi(t)|}\,d\lambda_F(t)
-2\log|\pfi(t)|+\log{2}
\notag\\
=U^{\lambda_F}&-G^{\lambda_F}_E(z)
+\int_F\log|1-\pfi(z)\pfi(t)|\,d\lambda_F(t)\notag\\
&-2\log|\pfi(z)|+\log{2},
\label{95}
\end{align}
uniformly inside the domain $D$, where
$$
G^{\lambda_F}_E(z)=\int_F\log\biggl|\frac{1-\pfi(z)\pfi(t)}{\pfi(z)-\pfi(t)}\biggr|
\,d\lambda_F(t)
$$
is the Green potential (with respect to the domain~$D$) of the measure $\lambda_F\in M_1(F)$
(see Appendix, \S~\ref{s4}, Lemma~\ref{lem2}).

A direct consequence of \eqref{95} and Lemma~\ref{lem2} is that the limit distribution of the zeros of
the polynomials $q_{2n}$ does exist, the equality
\begin{equation}
\mu=\frac14\beta_E(\lambda_F)+\frac34\tau_E
\label{96}
\end{equation}
holding for the corresponding limit measure $\mu\in M_1(E)$.

\section{Relation to Stahl's results. Some concluding remarks}\label{s3}

\subsection{}\label{s3s1}
In this section, we shall use the notation $\Delta:=[-1,1]$.
It what follows it will be assumed that $f_1=f$, $f_2=f^2$ in definition \eqref{4}--\eqref{5}
for Hermite--Pad\'e polynomials of type~II, where $f\in\sZ(\Delta)$
(see \eqref{701}).

For an arbitrary $n\in\NN$, the
Hermite--Pad\'e polynomials of type~II associated with the collection of
functions $[1,f,f^2]$ are defined in a~standard way from the relation
\begin{equation}
R_n(z):=(Q_{n,0}+Q_{n,1}f+Q_{n,2}f^2)(z)=O\(\frac1{z^{2n+2}}\),
\quad z\to\infty.
\label{709}
\end{equation}

Following Stahl~\cite{Sta88} (see also~\cite{BaGr96}, \S\,8.6) we shall use the following
definitions and notation, which are adapted appropriately to those introduced above.

Let $\RS$ be a Riemann surface and let $\pi\colon\RS\to\myh{\CC}$ be the
canonical projection. By this, the Riemann surface
$\RS$ is defined as a~multisheeted covering of the Riemann sphere; in what follow, we shall assume that
this covering is infinitely-\allowbreak sheeted.
Thus, the set $\pi^{-1}(z)$, $z\in\myh{\CC}\setminus\Sigma$, consists of a~finite number of points lying
on the Riemann surface~$\RS$ ``above'' the point~$z$. Here $\Sigma\subset\CC$ is
the finite (by the assumption) set of critical values of the projection~$\pi$.
So, for all points $z\in\myh{\CC}$ (except for the finite set~$\Sigma$) the mapping~$\pi$ is
locally biholomorphic. Since the covering~$\pi$ is infinitely-sheeted, it follows that the covering multiplicity condition
(as imposed by Stahl in~\cite{Sta88}) is always satisfied.
Consequently, we are thus able to compare the results obtained heuristically by Stahl
in~\cite{Sta88} with our results from \cite{Sue18},~\cite{Sue19b}, and from the present paper.
It what follows it will be assumed that $\infty\notin\Sigma$.

The functions
\begin{equation}
f(z):=\prod_{j=1}^p\(A_j-\frac1{\pfi(z)}\)^{\alpha_j},\quad z\in
D:=\myh\CC\setminus \Delta,
\label{701}
\end{equation}
provide an illustrative example of multivalued analytic functions;
here $p\geq2$, all $A_j\in\CC$ are pairwise distinct, $|A_j|>1$,
$\alpha_j\in\CC\setminus\QQ$, $j=1,\dots,p$, and $\alpha_1+\dots+\alpha_p=0$,
and, as before, we choose the branch of the function $(\cdot)^{1/2}$ such that
$(z^2-1)^{1/2}/z\to1$ as $z\to\infty$.
Following~\cite{Sue18c},~\cite{Sue19}, we denote by $\sZ(\Delta)$
the class of functions of the form~\eqref{701}.
The function $f\in\sZ(\Delta)$ is a~multivalued analytic function with the branch set
$\Sigma=\Sigma(f)=\{\pm1,a_j,j=1,\dots,p\}$, where $a_j:=(A_j+1/A_j)/2$. Moreover,
for the branch of $f(z)$, $z\in D$, of this function, as given by representation~\eqref{701},
we have $f(\infty)=1$.
The germ $f\in\HH(\infty)$ of a~function $f\in\sZ(\Delta)$ extends holomorphically to the domain~$D$ and
the closed interval $\Delta=S$ is the Stahl compact set for $f\in\HH(\infty)$. The corresponding
two-sheeted Stahl surface $\RS_2$ associated with~$f$ is the Riemann surface
$\RS_2=\RS_2(w)$ of the function $w^2=z^2-1$. A~point~$\zz$ on the Riemann surface $\RS_2(w)$ is the pair
$(z,w)$, where $w=\pm(z^2-1)^{1/2}$. This case corresponds to Pad\'e polynomials, and to the parameter $m=1$
in Stahl's notation of~\cite{Sta88}.

It is clear that the Riemann surface of an arbitrary function $f$ from the class $\sZ(\Delta)$
satisfies all the properties mentioned above; in particular, $\infty\notin\Sigma$.
Using this class as an example, we shall compare our results with those of Stahl~\cite{Sta88}.
In particular, below we shall discuss some numerical examples of distribution of the zeros of Hermite--Pad\'e polynomials
for functions of the form~\eqref{701}.

By $\infty^{(0)}$ we shall denote the point of the set $\pi^{-1}(\infty)\subset\RS$
at which the function~$f$ (which was initially defined by representation~\eqref{701}) assumes the value~$1$.
Following a~common tradition, we shall identify this point
with the point~$\infty$ on the Riemann sphere~$\myh\CC$.

\subsection{}\label{s3s2}
The existence of the limit distribution of the zeros of the Pad\'e polynomials $P_{n,j}$, $j=1,2$,
for an arbitrary function~$f$ from the class $\sZ(\Delta)$ is provided
by Stahl's theorem.
Moreover,
$$
\frac1n\chi(P_{n,j})\overset{*}\longrightarrow d\tau_\Delta(x)
:=\frac1\pi\frac{dx}{\sqrt{1-x^2}},
\quad n\to\infty.
$$

Following Stahl \cite{Sta88}, we set
\begin{equation}
H(z,x):=
\begin{cases}
z-x, & \ |x|\leq1,\\
(z-x)/|x|, & \ |x|>1,\\
1,& \ x=\infty.
\end{cases}
\label{702}
\end{equation}
For an arbitrary measure $\mu$, $S(\mu)\subset\CC$, following
formula (3.1) of~\cite{Sta88}, we define the spherically normalized\footnote{Note that in~\cite{Chi18}
the spherical normalization of a~potential is defined differently.}
potential:
\begin{equation}
p(\mu;z):=\int\log|H(z,x)|\,d\mu(x).
\label{703}
\end{equation}
By $\mathscr D$ we denote the class of all domains~$\mathfrak D$
on the Riemann surface $\RS=\RS(f)$ of the function $f\in\sZ(\Delta)$, ${\mathfrak D}\subset\RS$, such that
$\infty^{(0)}\in {\mathfrak D}$ and there exists the Green function $g_{\mathfrak D}(\zz,\mytt)$ for the domain~${\mathfrak D}$
(it is therefore assumed that the boundary $\partial {\mathfrak D}$ of the domain ${\mathfrak D}$ is a~nonpolar set; see
\cite{Chi18},~\cite{Chi19}).
In what follows, we assume that $g_{\mathfrak D}(\zz,\mytt)\equiv0$ for $\zz\in {\mathfrak D}$ for $\mytt\in\RS\setminus{{\mathfrak D}}$.

Now let $\nu$ be a~signed measure on the Riemann surface $\RS$ with finite total mass
(see \cite{Chi18}) and $\mathfrak D\in\mathscr D$. We define the Green
potential of the signed measure~$\nu$ as follows (see \cite{Chi18},~\cite{Chi19}):
\begin{equation}
G_{\mathfrak D}(\nu;\zz):=\int g_{\mathfrak D}(\zz,\mytt)\,d\nu(\mytt).
\label{704}
\end{equation}
Following~\cite{Sta88}, given a~point $\zz\in\RS$ and a~measure~$\mu$ on~$\myh{\CC}$, we set
\begin{equation}
p_\mu(\zz):=p(\nu;\pi(\zz)),\quad z\in\RS.
\label{705}
\end{equation}
So, $p_\mu$ is the lifting of the potential $p(\mu;z)$ of the measure~$\mu$ to the Riemann surface~$\RS$.
For a~fixed arbitrary integer number $m\ge1$, the following two
functions of a~point $\zz\in\RS$ were introduced in~\cite{Sta88}:
\begin{align}
r(\zz)=r(\mu,\mathfrak D;\zz):&=G_{\mathfrak
D}((m+1)\delta_{\infty^{(0)}}-\pi^{-1}(\mu));z)
\notag\\
&=(m+1)g_{\mathfrak D}(\zz,\infty^{(0)})
-\int g_{\mathfrak D}(\zz,\mytt)\,d\mu(\pi(\mytt)),
\label{706}\\
d(\zz):&=r(\zz)-p_\mu(\zz).
\end{align}
In~\cite{Sta88} it is noted that the case $m=1$ corresponds to Pad\'e polynomials, and hence
this case is classical. Therefore, the first nontrivial case, which
appears for $m=2$, corresponds to Hermite--Pad\'e polynomials of types~I and~II.
Note that in~\cite{Sue18}, for $m=3$, new Hermite--Pad\'e polynomials (intermediate between the
Hermite--Pad\'e polynomials of types~I and~II) were introduced.
In the present paper, we are concerned only with the case $m=2$, because we are dealing
here only with a~pair of functions, rather than with a~triple of functions, as
in~\cite{Sue18d}.
That is, we discuss the features
of the first nontrivial step (after the Pad\'e polynomials)
in the study of the asymptotic properties of Hermite--Pad\'e polynomials. In
this case, representation~\eqref{706} for the function $r(\zz)$ assumes the form
\begin{equation}
r(\zz)=3g_{\mathfrak D}(\zz,\infty^{(0)})
-\int g_{\mathfrak D}(\zz,\mytt)\,d\mu(\pi(\mytt)).
\label{708}
\end{equation}
The function $r(\zz)$ thus defined has the following properties:

\begin{enumerate}
\item
$r(\zz)=0$ quasi everywhere on $\partial\mathfrak D$ (we recall that the set $\partial\mathfrak D$ is nonpolar);
\item
$r(\zz)=3\log|\pi(\zz)|+O(1)$ as $\zz\to\infty^{(0)}$;
\item
the function $r(\zz)-\log|\pi(\zz)|$ is harmonic in the domain ${\mathfrak D}\setminus\infty^{(0)}$.
\end{enumerate}

The first Stahl result (\cite{Sta88}, Theorem~3.2) in this direction is as follows.

\begin{statement}\label{stat1}
Let $f\in\sZ(\Delta)$. Then there exist a~unique domain
${\mathfrak D}_1\in\mathscr D$ and a~unique measure $\nu_1$,
$S(\nu_1)\subset\myh{\CC}$, such that the function $r(\zz)=r(\nu_1,{\mathfrak D}_1;\zz)$ can be harmonically extended from
the domain~$\mathfrak D_1$ to some neighborhood of~$\partial {\mathfrak D}_1$.
\end{statement}

The domain ${\mathfrak D}_1\in\mathscr D$ and the measure $\nu_1\in
M_1(\myh{\CC})$ are called by Stahl the convergence domain of type~I
and the asymptotic distribution of type~I, respectively (see
\cite{Sta88}, Definition 3.3). Moreover, Stahl claims that the limit
distribution of the zeros of Hermite--Pad\'e polynomials of type~I does
exist and coincides with the measure $\nu_1$ (see Theorem~4.3
in~\cite{Sta88}):
\begin{equation}
\frac1n\chi(Q_{n,j})\overset{*}\longrightarrow \nu_1,
\quad j=0,1,2,\quad n\to\infty.
\label{710-1}
\end{equation}
In the domain $D_1:=\pi({\mathfrak D}_1)$, the limit relation holds
\begin{equation}
\lim_{n\to\infty}|Q^{*}_{n,j}(z)|^{1/n}=e^{p(\nu_1;z)},
\quad z\in\pi({\mathfrak D}_1),\quad j=0,1,2,
\label{711}
\end{equation}
where $Q^{*}_{n,j}$ are the spherically normalized\footnote{The spherical normalization of polynomials
(see, for example, \cite{GoRa87}) is quite similar to that of potentials~\eqref{702}.}
polynomials
($\mdeg{Q^{*}_{n,j}}/n\to1$) and the convergence
in \eqref{711} is understood as the
convergence in capacity on compact subsets of the domain
$\pi({\mathfrak D}_1)$.

We now set
\begin{equation}
\mathfrak D_0:=\bigl\{\zz\in \mathfrak D_1: r(\zz)>r(\mytt)\ \text{for all}\
\mytt\in \mathfrak D_1\ \text{such that}\ \pi(\mytt)=\pi(\zz),\ \mytt\neq\zz\bigr\}.
\label{710-2}
\end{equation}
Hence (see \cite{Sta88}) the set $\mathfrak D_0$ is a~subdomain of~$\mathfrak
D_1$. Since $r(\zz)=3\log|\pi(\zz)|+O(1)$ as $\zz\to\infty^{(0)}$ and since
$r(\zz)-\log|\pi(\zz)|$ is a~harmonic function in
$\mathfrak D_1\setminus\{\infty^{(0)}\}$,
we have $\infty^{(0)}\in \mathfrak D_0$.
The set $\pi(\mathfrak D_0)$ is a~domain on
$\myh{\CC}$ such that $\infty^{(0)}\in \pi(\mathfrak D_0)$ and
the set $\myh{\CC}\setminus\pi(\mathfrak D_0)$ consists of a~finite number of analytic arcs.
Let $\partial/\partial n$ be the
derivative with respect to the inner normal on $\partial\mathfrak D_0$. We set
\begin{equation}
d\mu_0(\zz):=\frac1{2\pi}\frac{\partial}{\partial n}
r(\nu_1,\mathfrak D_1;\zz)\,ds_{\zz},\quad \zz\in\partial \mathfrak D_0,
\label{711-1}
\end{equation}
where $ds_{\zz}$ is the element of arc length corresponding to a~point $\zz\in\partial
\mathfrak D_0$. Since the function $r(\zz)$ (see~\eqref{708}) has a~logarithmic pole of order~3
at the point $\zz=\infty^{(0)}$ and is harmonic
in $\mathfrak D_0\setminus\{\infty^{(0)}\}$, it follows that $\mu_0$ is a~positive measure supported
in $\partial \mathfrak D_0$ and $\mu_0(\partial \mathfrak D_0)=3$.
We set
$$
\nu_2=\frac13\pi(\mu_0).
$$
According to Stahl, the domain $D_2:=\pi(\mathfrak D_0)$ and the measure $\nu_2$ are called the convergence domain of type~II
and the asymptotic distribution of type~II, respectively
(see \cite{Sta88}, Definition~3.6). Another Stahl's result
(see \cite{Sta88}, Theorem~4.5) is to the effect that the limit distribution of the zeros
of Hermite--Pad\'e polynomials of type~II does exist and coincides with the measure $\nu_2$.
Namely, the following result holds.

\begin{statement}\label{stat2}
Under the hypotheses of Statement~\ref{stat1},
\begin{equation}
\frac1{2n}\chi(q_{2n})\overset{*}\longrightarrow\nu_2,
\quad n\to\infty.
\label{713}
\end{equation}
Moreover,
\begin{equation}
\lim_{n\to\infty}|q_{2n}^{*}(z)|^{n}=e^{p(\nu_2;z)},
\quad z\in\pi(\mathfrak D_0);
\label{714}
\end{equation}
the convergence in \eqref{714} is understood as the
convergence in capacity on compact subsets of
the domain $\pi(\mathfrak D_0)$, $q_{2n}^{*}$ is the spherically normalized
Hermite--Pad\'e polynomial of type~II.
\end{statement}

From Stahl's point of view~\cite{Sta88}, the domain $\mathfrak D_0$ is the zero\footnote{According
to an established tradition, the numbering of sheets of a~Riemann surface
starts from the zero (open) sheet, which as a~rule is identified with the
``physical'' plane~$\myh{\CC}$.} sheet of the Riemann surface $\RS(f)$,
and the open set $\mathfrak D_1\setminus\myo{\mathfrak D_0}$ is its first sheet.

\subsection{}\label{s3s3}
We now show that the above Stahl's results~\cite{Sta88}, the results of
\cite{Sue18}, \cite{Sue19}, \cite{Sue19b}, and the results of the present paper
compare favorably with each other.

As in~\cite{Sue19}, we assume that in representation \eqref{701} the quantities
$A_j$ and $\alpha_j$, $j=1,\dots,p$, are chosen so that, for each $j=1,\dots,p$,
we have $A_j=\myo{A}_k$ for some $k\in\{1,\dots,p\}$, the corresponding exponents
being equal: $\alpha_j=\alpha_k\in\RR\setminus\QQ$.
It is clear that with these conditions on the parameters $A_j$ and $\alpha_j$ the function $f$ is still
an infinite-valued function. Hence the conclusions of Theorems~1 and~2 hold (these results were announced
in~\cite{Sue19}; see also~\cite{RaSu13}).
According to Theorem~2 of~\cite{Sue19}, there exists a~compact set $\FF\subset\RS(f)$ such that
the component of the set $\RS(f)\setminus\FF$ which contains the point $\zz=\infty^{(0)}$
is a~two-sheeted domain over~$\myo\CC$, whose boundary
coincides with the compact set~$\FF$ and consists of a~finite number of analytic arcs.
We denote this domain by $\mathfrak G_1$, $\partial \mathfrak G_1=\FF$.

Let, as before (see \S\,\ref{s1s3}),
$V(\zz)=\log|z-w|=-\log|\Phi(\zz)|$, $\zz\in \mathfrak G_1$,
$$
P^{\lambda_{\FF}}(\zz)=
\int_{\FF}\log\frac{\left|1-1/\bigl(\Phi(\zz)\Phi(\mytt)\bigr)\right|}
{|z-t|^2}\,d\lambda_{\FF}(\mytt),\quad
\zz\in \mathfrak G_1\setminus(F^{(0)}),
$$
where $\lambda_{\FF}\in M_1(\FF)$ is the equilibrium measure for
$P^{\lambda_\FF}(\zz)+V(\zz)$ (this is the potential of an external field; i.e.,
$P^{\lambda_\FF}(\zz)+V(\zz)\equiv w_\FF=\const$ for $\zz\in\FF$).

We set $u(\zz):=P^{\lambda_\FF}(\zz)+V(\zz)-w_\FF$, $\zz\in \mathfrak G_0$.
The function $u(\zz)$ thus defined has the following properties (see \cite{Sue19},~\cite{Sue19b}):
\begin{enumerate}
\item
$u(\zz)\equiv0$ for $\zz\in\partial \mathfrak G_0=\FF$;
\item
$\dfrac{\partial u}{\partial n^{+}}(\zz)
=\dfrac{\partial u}{\partial n^{-}}(\zz)$, $\zz\in\FF\setminus{e}$,
$\#{e}<\infty$,
where $\partial/\partial n^{\pm}$ are the normal derivatives at a~point $\zz\in\FF$
taken from the both sides of~$\FF$;
\item $\displaystyle
-\frac1{2\pi}dd^cu=\lambda_\FF
+\lambda_F-3\delta_{\infty^{(0)}};$
\item
$u(\zz)=-3\log|z|+O(1)$ as $\zz\to\infty^{(0)}$.
\end{enumerate}
It follows that
$$
u(\zz)=-r(\lambda_\FF,\mathfrak G_1;\zz)=-r(z),
$$
and hence $\lambda_{\pi(\FF)}=\nu_1$ and $\mathfrak G_1=\mathfrak D_1$.

Let us now consider the set
$$
\mathfrak G_0:=\{\zz\in \mathfrak G_1: u(\zz)<u(\mytt),\mytt\in \mathfrak G_1,\pi(\mytt)=\pi(\zz),\mytt\neq\zz\}.
$$
According to \S~2.1 of \cite{Sue19b} and from the above properties of the function $u(\zz)$
it follows that $\mathfrak G_0$ is a~one-sheeted (with respect to the projection~$\pi$) subdomain of the domain
$\mathfrak G_1$, containing the point $\zz=\infty^{(0)}$ and whose boundary is
$\pi^{-1}(\Delta)\cap\mathfrak G_1$. In other words,
$\mathfrak G_0=\{\zz=(z,(z^2-1)^{1/2}),z\in D=\myh{\CC}\setminus{\Delta}\}$.
We have $u(\zz)=-r(\zz)$, and so $\mathfrak G_0=\mathfrak D_0$. It now easily follows that
the measure $\nu_2$, as introduced by Stahl~\cite{Sta88}, coincides with the measure~\eqref{38}.

So, the results obtained in~\cite{Sue18},~\cite{Sue19},~\cite{Sue19b} and in the present paper
are quite in agreement
with those obtained by Stahl in~\cite{Sta88}.

\subsection{}\label{s3s4}
Finally, we conclude by discussing in what direction a~further
development of the theory of Hermite--Pad\'e polynomials associated
with multivalued analytic functions is possible. Let us emphasize once
again that Stahl's results of~\cite{Sta88} are heuristic and require
rigorous justification. Hence they should be considered as conjectures
in the same way as those proposed by Nuttall in~\cite{Nut84}.

The advantage of the scalar approach over the vector one
can be visualized as follows.
As in the case of Pad\'e polynomials, here we prove the existence of a~single extremal compact set
(but which now lies on the Riemann surface); this can also be clearly seen from numerical results
on evaluation of zeros of Hermite--Pad\'e polynomials of type~I and~II for functions of the form~\eqref{701} and
for slightly more general functions. Of course, in Examples~\ref{ex1}--\ref{ex3} that follow
all $\alpha_j$ in representation \eqref{701} are from~$\QQ$.

\begin{example}\label{ex1} In this example, $p=3$, $A_1\in\RR$,
$A_2=\myo{A}_3\notin\RR$, $\alpha_1=-2/3$, $\alpha_2=\alpha_3=1/3$ in representation \eqref{701}.
In Fig.~\ref{fig_1}, the dark (blue, red, and black) points are the zeros of
the Hermite--Pad\'e polynomial of type~I $Q_{n,0}$, $Q_{n,1}$ and $Q_{n,2}$, which simulate
the $S$-compact set $\FF$ (or, more precisely, its projection $\pi(\FF)$).
The arrangement of these points resembles the Chebotarev compact set of minimal capacity for three points (which is depicted
after the change $z\mapsto 1/z$).
The extremal property of our compact set~$\FF$ is related to the point $\zz=\infty^{(0)}$,
which lies on the zero sheet of the Riemann surface $\RS(f)$. The zeros of the Hermite--Pad\'e polynomial of type~II
$q_{2n}$ (light blue points) simulate the second $S$-compact set $\mybfe$. Since the initial data are symmetric
about the real line, we have here $\pi(\mybfe)=[-1,1]$.
However, this is not so in the general case, as we shall see from Examples \ref{ex2} and~\ref{ex3}.
\end{example}

\begin{example}\label{ex2}
In this example, $f\in\sZ(\Delta)$ is such that $p=4$ in representation \eqref{701}, and besides,
all the pairwise distinct points $a_1,a_2,a_3,a_4$ lie in the upper half-plane.
The parameters $\alpha_1,\alpha_2,\alpha_3,\alpha_4$
are chosen so that the $S$-compact set~$\FF$ is a~continuum (connecting these four points).
In Fig.~\ref{fig_2},
 the dark (blue, red, and black) points are the zeros
of the Hermite--Pad\'e polynomials of type~I\enskip $Q_{n,0}$, $Q_{n,1}$ and $Q_{n,2}$,
which simulate the
compact set~$\FF$ (or, more precisely, its projection $\pi(\FF)$). As in the first example, the complement of~$\FF$
(which contains the point $\zz=\infty^{(0)}$) is the Stahl domain
$\mathfrak D_1$ on the Riemann surface $\RS(f)$. The compact set $\pi(\FF)$ is an analogue of the classical
Chebotarev continuum for 4~points. In addition to four branch points, it also contains two Chebotarev points
at which the equilibrium measure~$\lambda_F$
has zero density. The light blue points are the zeros
of the Hermite--Pad\'e polynomials of type~II $q_{2n}$, which simulate the second $S$-compact set~$\mybfe$.
In our case, $\pi(\mybfe)\neq \Delta$. More precisely,
$\pi(\mybfe)\cap \Delta=\{-1,1\}$. The complement of~$\mybfe$ on the Riemann surface~$\RS(f)$
(which contains the point $\zz=\infty^{(0)}$) is the Stahl domain $\mathfrak D_0$ (the zero sheet).
The open set $\mathfrak D_1\setminus\myo{\mathfrak D}_0$ is the first sheet (in our setting, this set is connected, and hence, is a~domain).
But this is not always the case\,---\,corresponding examples can be constructed for a~slightly more general
class of functions than the class $\sZ(\Delta)$ (see Example \ref{ex4} below).
\end{example}

\begin{example}\label{ex3}
In this example, $f\in\sZ(\Delta)$ is such that we again have
$p=4$ in representation \eqref{701}, the four points $a_1,a_2,a_3,a_4$ are such that there is again no symmetry about the real line.
Besides, $\alpha_j\in\{-1/2,1/2\}$
(the condition $\sum_{j=1}^4\alpha_j=0$ is retained). So, $f$ has branch points of order~2 only.
Geometrically, the four points
$a_j$ were chosen so that the projection $\pi(\FF)$ of the $S$-compact set~$\FF$ consists of two disjoint arcs.
In Fig.~\ref{fig_3},
 the dark (blue, red, and black) points are the zeros of the Hermite--Pad\'e polynomials of type~I $Q_{n,0}$, $Q_{n,1}$ and $Q_{n,2}$,
 which simulate the projection $\pi(\FF)$ of the compact set~$\FF$.
The complement of~$\FF$ (which contains the point $\zz=\infty^{(0)}$) is the Stahl domain
$\mathfrak D_1$ on the Riemann surface $\RS(f)$.
The light blue points represent the zeros of the Hermite--Pad\'e polynomials of type~I, which simulate the second $S$-compact
set~$\mybfe$. In the case under consideration, $\pi(\mybfe)\neq \Delta$. More precisely,
$\pi(\mybfe)\cap \Delta=\{-1,1\}$. The complement of~$\mybfe$ on the Riemann surface~$\RS(f)$ (which contains the
point $\zz=\infty^{(0)}$) is the Stahl domain $\mathfrak D_0$ (the zero sheet). The open
set $\mathfrak D_1\setminus\myo{\mathfrak D}_0$ is the first sheet (in our case, this set is connected, and hence, a~domain).
\end{example}

\begin{example}\label{ex4}
In this example, a more general class of multivalued analytic functions than the class $\sZ(\Delta)$ is considered. Namely,
instead of one closed interval $\Delta=[-1,1]$, we consider two
disjoint real closed intervals $\Delta_1$
and $\Delta_2$ and the corresponding two inverses of the Joukowsky function
$\pfi_{\Delta_1}(z)$ and $\pfi_{\Delta_2}(z)$. In this case, the multivalued function~$f$
has the form
\begin{equation}
f(z):=\prod_{j=1}^p\(A_j-\frac1{\pfi_{\Delta_1}(z)}\)^{\alpha_j}
\prod_{k=1}^q\(B_k-\frac1{\pfi_{\Delta_2}(z)}\)^{\beta_k},
\label{715}
\end{equation}
where, as before, it is assumed that the quantities $A_j$ are pairwise disjoint,
$|A_j|>1$, and the corresponding quantities $a_j$, $\pfi_{\Delta_1}(a_j)=A_j$, are such that
$a_j\notin(\Delta_1\cup\Delta_2)$. Similar assumptions are made about the quantities $B_k$.
Moreover, it is assumed that $a_j\neq
b_k$ for all~$j$ and that the points $a_j$, $b_k$ and the exponents $\alpha_j$, $\beta_k$ are such that
geometrically the problem is symmetric about the real line. In this case, the
results of~\cite{RaSu13} can be applied
(the results of~\cite{Sue19} are inapplicable, because this paper deals only with the class $\sZ(\Delta)$.
Under these conditions, the $S$-compact set~$\FF$ should be such that the
compact set $\pi(\FF)$ is symmetric about the real line, while the
second $S$-compact set~$\mybfe$ should be such that $\pi(\mybfe)=\Delta_1\cup \Delta_2$.
This is fully confirmed by numerical results. In Fig.~\ref{fig_4},
the dark (blue, red, and black) points are the zeros of the Hermite--Pad\'e polynomials of type~I $Q_{n,0}$,
$Q_{n,1}$ and $Q_{n,2}$, which simulate the compact set $\pi(\FF)$.
The complement of~$\FF$ (which contains the point $\zz=\infty^{(0)}$) is the Stahl domain
$\mathfrak D_1$ on the Riemann surface $\RS(f)$. The light blue points are the zeros
of the Hermite--Pad\'e polynomials of type~II $q_{2n}$, which simulate the compact set
$\pi(\mybfe)$.
In this case, $\pi(\mybfe)=\Delta_1\cup \Delta_2$.
The complement of~$\mybfe$ on the Riemann surface $\RS(f)$ (which contains the
point $\zz=\infty^{(0)}$) is the Stahl domain $\mathfrak D_0$ (the zero sheet).
The open set $\mathfrak D_1\setminus\myo{\mathfrak D}_0$ is the first sheet (in this case, this set is
disconnected, and hence is not a~domain). The open set
$\myo\CC\setminus\pi(\FF)$ consists of two domains, and hence, is not a~domain.
Numerical calculations have been performed with
$p=2$, $A_1=\myo{A_2}\notin\RR$ and $q=3$, $B_1=\myo{B}_2\notin\RR$, $B_3\in\RR$ in~\eqref{715}.
The number of Chebotarev points of zero density is five.
\end{example}

It should be noted that in all above Examples~\ref{ex1}--\ref{ex4} the compact set $\pi(\FF)$ is disjoint from
the compact set $\pi(\mybfe)$ and from the closed interval~$\Delta$. This is a~certain heuristic principle,
which always holds in the function class $\sZ(\Delta)$ (and even in the more general function class
of the form~\eqref{715}).

\section{Appendix}\label{s4}

\begin{lemma}\label{lem2}
Let $\mu\in M_1(F)$ and let
$$
v(z)=v(z;\mu):=\int\log\bigl|1-\pfi(z)\pfi(t)\bigr|\,d\mu(t).
$$
Then
\begin{equation}
v(z;\mu)=-\frac12U^{\beta_E(\mu)+\tau_E}(z)+\const,
\label{88.2}
\end{equation}
where $\beta_E(\mu)$ is the balayage of the measure~$\mu$ from the domain~$D$ to~$E$ and $\tau_E$
is the Chebyshev measure of the interval~$E$.
\end{lemma}
\begin{proof}
Indeed, since the compact sets $E$ and~$F$ lie on the real line~$\RR$ and since the function $\pfi(z)$
assumes real values for $z\in\RR\setminus E$, we have, for the Green function $g_E(z,t)$ of the domain~$D$,
\begin{equation}
g_E(z,t)
=\log\frac{|1-\pfi(z)\myo{\pfi(t)}|}{|\pfi(z)-\pfi(t)|}
=\log\frac{|1-\pfi(z)\pfi(t)|}{|\pfi(z)-\pfi(t)|}
\label{151}
\end{equation}
for $z,t\in\RR\setminus E$.
Let us now employ the identity
\begin{equation*}
z-a\equiv -\frac{\bigl(\Phi(\zz)-\Phi(\maa)\bigr)
\bigl(1-\Phi(\zz)\Phi(\maa)\bigr)}
{2\Phi(\zz)\Phi(\maa)},
\quad z,a\in D,
\end{equation*}
which can be easily verified (see \cite{GoSu04},~\cite{Sue18}).
Using this relation, representation~\eqref{151}, and the relation between the functions
$\Phi(\zz)$ and $\pfi(z)$, we finally get
\begin{equation}
g_E(z,t)=\log\frac{|1-\pfi(z)\pfi(t)|^2}{2|z-t|\cdot|\pfi(z)\pfi(t)|}.
\label{152}
\end{equation}
By identity \eqref{152}, for the Green potential $G^\mu_E(z)$ of the measure $\mu\in M_1(F)$ we get
\begin{align}
G^\mu_E(z):&=\int g_E(z,t)\,d\mu(t)=\notag\\
&=2\int\log|1-\pfi(z)\pfi(t)|\,d\mu(t)+U^\mu(z)-\log|\pfi(z)|+\const
\notag\\
&=2v(z;\mu)+U^\mu(z)-\gamma_E+U^{\tau_E}(z)+\const.
\label{88.3}
\end{align}
Let us apply the operator $\dd^c$ to the function $v(z;\mu)$. We have (see \cite{Chi18})
$$
\frac1{2\pi}\dd^c G^\mu_E(z)=\beta_E(\mu)-\mu,
$$
and hence, using \eqref{88.3}, this establishes
$$
2\pi(\beta_E(\mu)-\mu)
=2\dd^c v(z;\mu)+2\pi(\delta_\infty-\mu)+2\pi(\delta_\infty-\tau_E).
$$
It follows that
$$
2\dd^c v(z;\mu)=2\pi(\beta_E(\mu)+\tau_E-2\delta_\infty).
$$
Therefore,
$$
v(z;\mu)=-\frac12 U^{\beta_E(\mu)+\tau_E}(z)+\const.
$$
Lemma~\ref{lem2} is proved.
\end{proof}

\clearpage
\newpage
\begin{figure}[!ht]
\centerline{
\includegraphics[width=12cm,height=12cm]{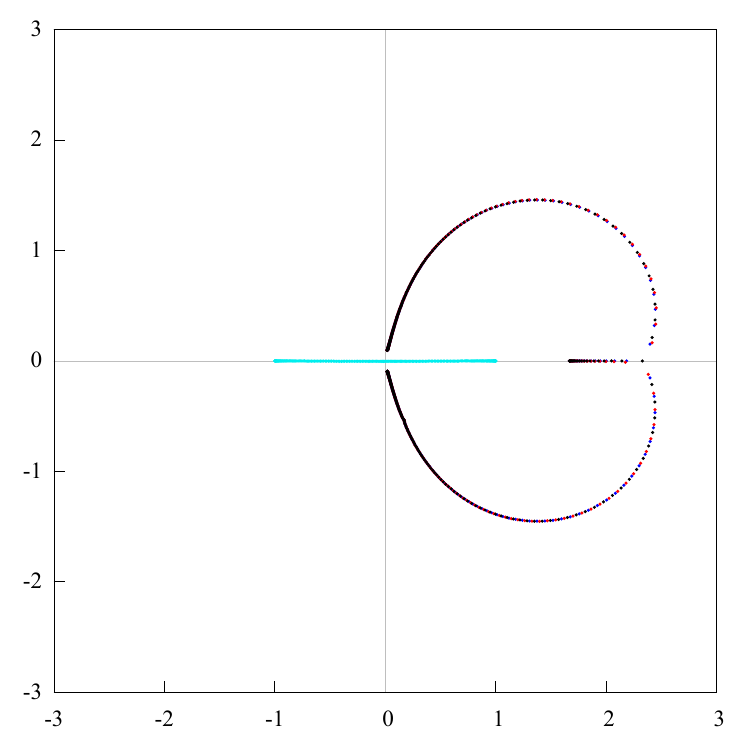}}
\vskip-6mm
\caption{
Here $p=3$, $A_1\in\RR$,
$A_2=\myo{A}_3\notin\RR$, $\alpha_1=-2/3$, $\alpha_2=\alpha_3=1/3$ in representation \eqref{701}.
The dark (blue, red, and black) points are the zeros of
the Hermite--Pad\'e polynomials of type~I\enskip
$Q_{350,0}$, $Q_{350,1}$ and $Q_{350,2}$, which simulate
the compact set $\pi(\FF)$.
The arrangement of these points resembles the Chebotarev compact set of minimal capacity for three points (which is depicted
after the change $z\mapsto 1/z$).
The extremal property of our compact set~$\FF$ is related to the point $\zz=\infty^{(0)}$,
which lies on the zero sheet of the Riemann surface $\RS(f)$. The zeros of the Hermite--Pad\'e polynomial of type~II
$q_{100}$ (light blue points) simulate the second compact set $\pi(\mybfe)$. Since the initial data are symmetric
about the real line, we have here $\pi(\mybfe)=[-1,1]$.
However, this is not so in the general case, as we shall see from Figs.~\ref{fig_2}
and~\ref{fig_3}.
}
\label{fig_1}
\end{figure}

\clearpage
\newpage
\begin{figure}[!ht]
\centerline{
\includegraphics[width=12cm,height=12cm]{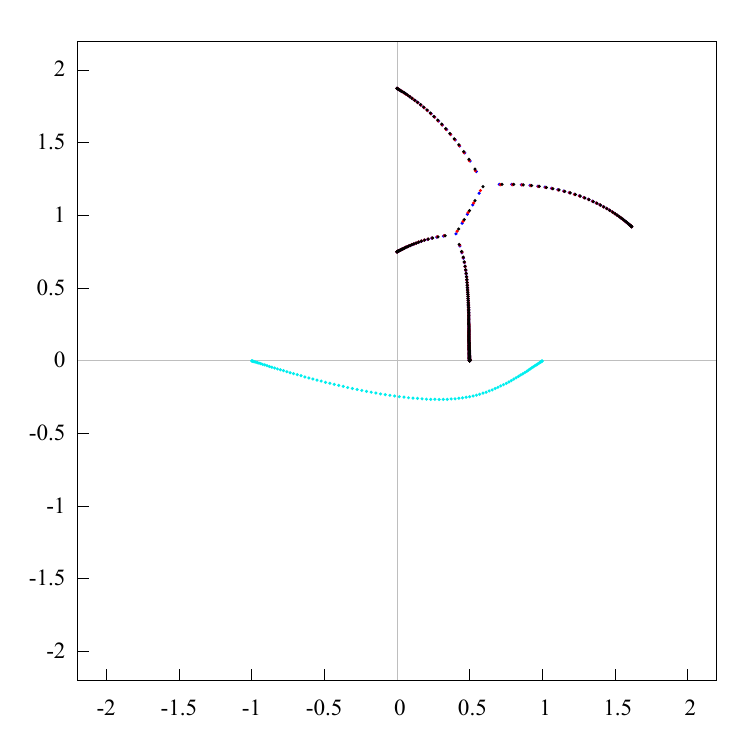}}
\vskip-6mm
\caption{
Here in \eqref{701} we take $p=4$.
All the four points $a_1,a_2,a_3,a_4$ lie in the upper half-plane,
$\alpha_1=\alpha_2=1/5$, $\alpha_3=\alpha_4=-1/5$.
The corresponding compact set~$\FF$ is a~continuum.
The dark (blue, red, and black) points are the zeros
of the Hermite--Pad\'e polynomials of type~I\enskip $Q_{200,0}$, $Q_{200,1}$ and $Q_{200,2}$.
These points simulate the compact set $\pi(\FF)$.
The compact set $\pi(\FF)$ is an analogue of the classical
Chebotarev continuum for 4~points.
In addition to four branch points, it also contains two Chebotarev points of zero density.
The light blue points are the zeros
of the Hermite--Pad\'e polynomials of type~II\enskip $q_{100}$, which simulate the second compact set
$\pi(\mybfe)$. In this case, $\pi(\mybfe)\cap \Delta=\{-1,1\}$.
}
\label{fig_2}
\end{figure}

\clearpage
\newpage
\begin{figure}[!ht]
\centerline{
\includegraphics[width=12cm,height=12cm]{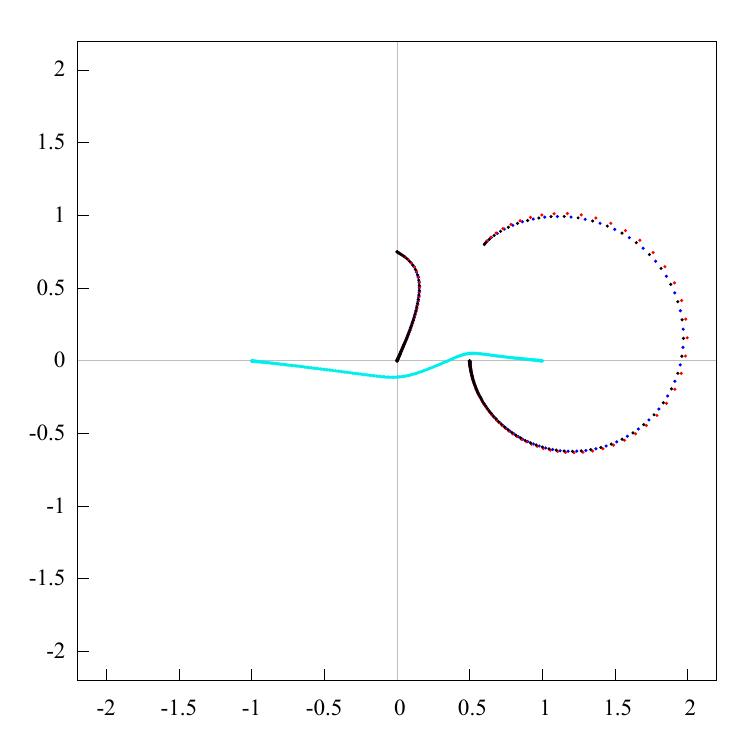}}
\vskip-6mm
\caption{
Here in \eqref{701} we take $p=4$. The four points $a_1,a_2,a_3,a_4$ are such that there is again no symmetry about the real line.
Besides, $\alpha_j\in\{-1/2,1/2\}$
(the condition $\sum_{j=1}^4\alpha_j=0$ is retained). So, $f$ has branch points of order~2 only.
Geometrically, the four points
$a_j$ were chosen so that the compact set $\pi(\FF)$ consists of two disjoint arcs.
The dark (blue, red, and black) points are the zeros of the Hermite--Pad\'e polynomials of type~I\enskip
$Q_{150,0}$, $Q_{150,1}$ and $Q_{150,2}$; these points simulate the
compact set $\pi(\FF)$.
The light blue points represent the zeros of the Hermite--Pad\'e polynomial of type~II\enskip $q_{200}$,
which simulate the compact set $\pi(\mybfe)$. In this case, $\pi(\mybfe)\cap
\Delta=\{-1,1\}$.
}
\label{fig_3}
\end{figure}

\clearpage
\newpage
\begin{figure}[!ht]
\centerline{
\includegraphics[width=12cm,height=12cm]{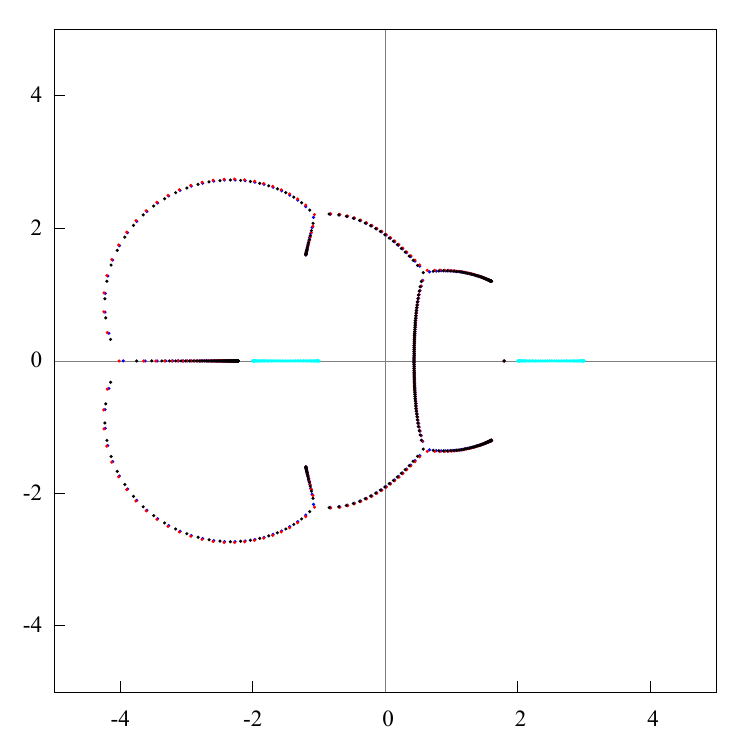}}
\vskip-6mm
\caption{
In this numerical example, the class of multivalued analytic functions defined by representation
\eqref{715} is considered. We consider two
disjoint real closed intervals $\Delta_1$
and $\Delta_2$ and the corresponding two inverses of the Joukowsky function
$\pfi_{\Delta_1}(z)$ and $\pfi_{\Delta_2}(z)$.
It is assumed that the points $a_j$, $b_k$ and the exponents $\alpha_j$, $\beta_k$
are such that
geometrically the problem is symmetric about the real line. In this case, the compact set
$\pi(\FF)$ should be symmetric about the real line, and the second compact set is $\pi(\mybfe)=\Delta_1\cup \Delta_2$.
This is fully confirmed by numerical results.
The dark (blue, red, and black) points are the zeros of the Hermite--Pad\'e polynomials of type~I
\enskip $Q_{350,0}$,
$Q_{350,1}$ and $Q_{350,2}$, which simulate the compact set $\pi(\FF)$.
The light blue points are the zeros
of the Hermite--Pad\'e polynomial of type~II $q_{100}$,
which simulate the compact set $\pi(\mybfe)$.
Numerical calculations have been performed with
$p=2$, $A_1=\myo{A_2}\notin\RR$ and $q=3$, $B_1=\myo{B}_2\notin\RR$, $B_3\in\RR$ in \eqref{715}.
The number of Chebotarev points of zero density is five;
this number agrees with that of the
branch points that differ from the points~$z=\pm1$.
}
\label{fig_4}
\end{figure}



\newpage
\begin{thebibliography}{99}

\bibitem{ApLy10}
\href{http://mi.mathnet.ru/eng/msb7515}
{A. I. Aptekarev, V. G. Lysov, ``Systems of Markov functions generated by graphs and
the asymptotics of their Hermite-Pad\'e approximants'', \textit{Sb. Math.}, \textbf{201}:2 (2010), 183--234.}

\bibitem{ApBoYa17}
\href{http://mi.mathnet.ru/eng/msb8632}
{A. I. Aptekarev, A. I. Bogolyubskii, M. Yattselev, ``Convergence of ray sequences of Frobenius-Pad\'e approximants'',
\textit{Sb. Math.}, \textbf{208}:3 (2017), 313--334}

\bibitem{BaGr96}
\href{http://www.ams.org/mathscinet-getitem?mr=1383091}
{Baker, George A., Jr.; Graves-Morris, Peter, \textit{Pad\'e approximants}, Second edition,
Encyclopedia of Mathematics and its Applications, \textbf{59}, Cambridge University Press,
Cambridge, 1996, xiv+746 pp. ISBN: 0-521-45007-1.}

\bibitem{BaGeLo18}
\href{http://mi.mathnet.ru/eng/sm8724}
{D. Barrios Rolan{\'\i}a, J. S. Geronimo, G. L\'{o}pez Lagomasino,
``High-order recurrence relations, Hermite-Pad\'e approximation and Nikishin systems'',
\textit{Sb. Math.}, \textbf{209}:3 (2018), 385--420.}

\bibitem{Chi18}
\href{http://mi.mathnet.ru/eng/tm3918}
{E. M. Chirka, ``Potentials on a compact Riemann surface'',
\textit{Proc. Steklov Inst. Math.}, \textbf{301} (2018), 272--303.}


\bibitem{Chi19}
\href{http://mi.mathnet.ru/eng/tm4007}
{E. M. Chirka, ``Equilibrium measures on a compact Riemann surface'',
\textit{Proc. Steklov Inst. Math.}, \textbf{306} (2019), in print.}

\bibitem{GoRa81}
\href{http://mi.mathnet.ru/eng/tm2391}
{A. A. Gonchar, E. A. Rakhmanov,
``On the convergence of simultaneous Pad\'e approximants for systems of functions of Markov type'',
\textit{Proc. Steklov Inst. Math.}, \textbf{157} (1983), 31--50.}

\bibitem{GoRa87}
\href{http://mi.mathnet.ru/eng/msb2759}
{A. A. Gonchar, E. A. Rakhmanov, ``Equilibrium distributions and degree of rational
approximation of analytic functions'', \textit{Math. USSR-Sb.}, \textbf{62}:2 (1989), 305--348.}

\bibitem{GoRaSo97}
\href{http://mi.mathnet.ru/eng/msb225}
{A. A. Gonchar, E. A. Rakhmanov, V. N. Sorokin,
``Hermite--Pad\'e approximants for systems of Markov-type functions'',
\textit{Sb. Math.}, \textbf{188}:5 (1997), 671--696.}

\bibitem{Gon03}
\href{http://mi.mathnet.ru/eng/spm4}
{A. A. Gonchar, ``Rational Approximations of Analytic Functions'',
\textit{Proc. Steklov Inst. Math.}, \textbf{272}, suppl. 2 (2011), S44--S57.}


\bibitem{GoSu04}
\href{http://mi.mathnet.ru/eng/spm8}
{A. A. Gonchar, S. P. Suetin, ``On Pad\'e approximants of Markov-type meromorphic
functions'', \textit{Proc. Steklov Inst. Math.}, \textbf{272}, suppl. 2 (2011), S58--S95.}

\bibitem{GoRaSu11}
\href{http://mi.mathnet.ru/eng/umn9452}
{A. A. Gonchar, E. A. Rakhmanov, S. P. Suetin,
``Pade\'e-Chebyshev approximants of multivalued analytic functions,
variation of equilibrium energy, and the S-property of stationary compact sets'',
\textit{Russian Math. Surveys}, \textbf{66}:6 (2011), 1015--1048.}

\bibitem{KoPaSuCh17}
\href{http://mi.mathnet.ru/eng/umn9786}
{A. V. Komlov, R. V. Palvelev, S. P. Suetin,
E. M. Chirka, ``Hermite-Pad\'e approximants for meromorphic functions on a compact
Riemann surface'', \textit{Russian Math. Surveys}, \textbf{72}:4 (2017), 671--706.}

\bibitem{LoLo18}
\href{https://doi.org/10.1016/j.jat.2017.10.002}
{A. L\'{o}pez-Garcia and G. L\'{o}pez Lagomasino, ``Nikishin systems on star-like sets: Ratio
asymptotics of the associated multiple orthogonal polynomials'',
\textit{J. Approx. Theory}, \textbf{225} (2018), 1--40.}

\bibitem{LoLo19}
A. L\'{o}pez-Garcia and G. L\'{o}pez Lagomasino, \textit{Nikishin systems on star-like sets: Ratio
asymptotic formulae for the associated multiple orthogonal polynomials}, 2019, 27 pp.,
\href{http://arxiv.org/abs/1907.03002}{arxiv:1907.03002}.


\bibitem{LoMeSz19}
\href{https://doi.org/10.1016/j.aim.2019.04.024}
{G. L\'{o}pez Lagomasino, S. Medina Peralta, J. Szmigielski, ``Mixed type Hermite-Pad\'{e}
approximation inspired by the Degasperis-Procesi equation'', 20 June 2019,
\textit{Advances in Mathematics}, \textbf{349}, 813--838.}

\bibitem{LoMi18}
\href{http://mi.mathnet.ru/eng/msb8768}
{A. L\'{o}pez-Garcia, E. Mina-Diaz, ``Nikishin systems on star-like sets: algebraic properties and weak
asymptotics of the associated multiple orthogonal polynomials'',
\textit{Sb. Math.}, \textbf{209}:7 (2018), 1051--1088.}

\bibitem{LoVa18}
\href{http://mi.mathnet.ru/eng/msb8889}
{Guillermo L\'{o}pez Lagomasino, Walter Van Assche, ``Riemann-Hilbert analysis for a Nikishin system'',
\textit{Sb. Math.}, \textbf{209}:7 (2018), 1019--1050.}

\bibitem{MaTs17}
\href{http://www.ams.org/mathscinet-getitem?mr=3598817}
{Toshiyuki Mano, Teruhisa Tsuda, ``Hermite-Pad\'{e} approximation, isomonodromic deformation
and hypergeometric integral'', \textit{Mathematische Zeitschrift}, \textbf{285}:1--2 (2017), 397--431.}

\bibitem{MaRaSu16}
\href{http://mi.mathnet.ru/eng/conm2}
{Andrei Mart{\'\i}nez-Finkelshtein, Evguenii A. Rakhmanov, Sergey P. Suetin,
``Asymptotics of type I Hermite-Pad\'{e} polynomials for semiclassical functions'',
Modern trends in constructive function theory, \textit{Contemp. Math.}, \textbf{661}, 2016, 199--228.}


\bibitem{Nik86}
\href{http://mi.mathnet.ru/eng/ivm7481}
{E. M. Nikishin, ``Asymptotic behavior of linear forms for simultaneous Pad\'e approximants'',
\textit{Soviet Math. (Iz. VUZ)}, \textbf{30}:2 (1986), 43--52.}

\bibitem{NiSo88}
\href{http://www.ams.org/mathscinet-getitem?mr=0953788}
{Nikishin, E. M., Sorokin, V. N.
\textit{Rational approximations and orthogonality}, Nauka, Moscow, 1988, 256 pp.}

\bibitem{Nut84}
\href{http://www.ams.org/mathscinet-getitem?mr=0769985}
{J. Nuttall, ``Asymptotics of diagonal Hermite--Pad\'e polynomials'', \textit{J. Approx.Theory},
\textbf{42} (1984), 299--386.}

\bibitem{RaSu13}
\href{http://mi.mathnet.ru/eng/msb8168}
{E. A. Rakhmanov, S. P. Suetin, ``The distribution of the zeros of the
Hermite-Pad\'e polynomials for a pair of functions forming a Nikishin system'',
\textit{Sb. Math.}, \textbf{204}:9 (2013), 1347--1390.}

\bibitem{Rak18}
\href{http://mi.mathnet.ru/eng/umn9832}
{E. A. Rakhmanov, ``Zero distribution for Angelesco Hermite-Pad\'{e} polynomials'',
\textit{Russian Math. Surveys}, \textbf{73}:3 (2018), 457--518.}


\bibitem{Sta87}
\href{http://www.ams.org/mathscinet-getitem?mr=0886690}
{H. Stahl, ``Three different approaches to a proof of convergence for Pad\'{e} approximants'',
\textit{Rational approximation and applications in mathematics and physics, {\L}a\'{n}cut, 1985},
Lecture Notes in Math., 1237, Springer, Berlin, 1987, 79--124.}

\bibitem{Sta88}
\href{http://www.ams.org/mathscinet-getitem?mr=1005350}
{H. Stahl, ``Asymptotics of Hermite--Pad\'e polynomials and related convergence results.
A summary of results'', \textit{Nonlinear numerical methods and rational approximation}
(Wilrijk, 1987), Math. Appl., \textbf{43}, Reidel, Dordrecht, 1988, 23--53; also the fulltext
preprint version is avaible, 79 pp.}

\bibitem{Sue18}
\href{http://mi.mathnet.ru/eng/tm3908}
{S. P. Suetin, ``On a new approach to the
problem of distribution of zeros of Hermite-Pad\'{e} polynomials for a Nikishin system'',
\textit{Proc. Steklov Inst. Math.}, \textbf{301} (2018), 245--261.}

\bibitem{Sue18b}
\href{http://mi.mathnet.ru/eng/umn9819}
{S. P. Suetin, ``Distribution of the zeros of Hermite-Pad\'{e} polynomials for a complex Nikishin system'',
\textit{Russian Mathematical Surveys}, \textbf{73}:2 (2018), 363--365.}

\bibitem{Sue18d}
Sergey P. Suetin, Hermite-Pad\'{e} polynomials and analytic continuation: new approach
and some results, 2018, 45 pp., \href{http://arxiv.org/abs/1806.08735}{arxiv:1806.08735}.


\bibitem{Sue18c}
\href{http://mi.mathnet.ru/eng/mzm12181}
{S. P. Suetin, ``On an Example of the Nikishin System'',
\textit{Math. Notes}, \textbf{104}:6 (2018), 905--914.}

\bibitem{Sue19}
\href{http://mi.mathnet.ru/eng/umn9884}
{S. P. Suetin, ``Existence of a three-sheeted Nutall surface for a certain class
of infinite-valued analytic functions'', \textit{Russian Math. Surveys}, \textbf{74}:2 (2019), 363--365.}

\bibitem{Sue19b}
\href{http://mi.mathnet.ru/eng/mz12451}
{S. P. Suetin, ``On the equivalence of scalar and vector equilibrium problems
for a pair of functions forming the Nikishin system'',
\textit{Math. Notes}, \textbf{106}:6 (2019), 904--916.}

\bibitem{Tri18}
\href{http://dx.doi.org/10.1561/3100000015}
{Antonio Trias, ``HELM: The Holomorphic Embedding Load-Flow Method:
Foundations and Implementations'',
Foundations and Trends in Electric Energy Systems, 3:3-4, 140--370.}

\bibitem{Van06}
\href{http://www.ams.org/mathscinet-getitem?mr=2247778}
{Van Assche, Walter, ``Pade and Hermite-Pad\'{e} approximation and orthogonality'',
\textit{Surv. Approx. Theory}, \textbf{2} (2006), 61--91.}


\end{thebibliography}
\end{document}